\documentclass[reqno]{amsart}
\usepackage{amssymb,amsmath,amsthm}
\usepackage{mathrsfs}
    \usepackage{a4wide}
\usepackage{amscd}
\usepackage{amsfonts}
\usepackage{amssymb}
\usepackage{latexsym}
\usepackage{color}
\usepackage{esint}
\usepackage{graphicx}
\usepackage{float}
\graphicspath{{Figures/}}

\usepackage{graphicx}
\usepackage[breaklinks=true,bookmarks=false]{hyperref}
\usepackage{cleveref}

\usepackage[makeroom]{cancel}

\newtheorem{theorem}{Theorem}

\newtheorem{definition}{Definition}
\newtheorem{lemma}{Lemma}
\newtheorem{proposition}[theorem]{Proposition}
\newtheorem{remark}{Remark}

\let\e=\varepsilon

\let\h=v

\let\p=\partial

\let\O=\Omega

\let\o=\omega

\numberwithin{equation}{section}

\let\hide\iffalse

\newcommand{\nablaA}{\nabla^{\text{a}}}
\newcommand{\nablaS}{\nabla^{\text{sym}}}

\newcommand{\R}{\mathbb{R}}
\renewcommand{\P}{\mathbf{P}}

\newcommand{\be}{\begin{equation}}
\newcommand{\bm}{\begin{multline}}
\newcommand{\ee}{\end{equation}}
\newcommand{\dd}{\mathrm{d}}

\newcommand{\xb}{x_{\mathbf{b}}}
\newcommand{\tb}{t_{\mathbf{b}}}

\newcommand{\Div}{\text{div} \, }

\newcommand{\Bes}{\begin{eqnarray*}}
\newcommand{\Ees}{\end{eqnarray*}}
\newcommand{\Be}{\begin{equation} }
\newcommand{\Ee}{\end{equation}}

\newtheorem{assumption}{Assumption}

\def\p{\partial}

\def\O{\Omega}
\def\R{\mathbb{R}}

\def\B{\begin{equation}}
\def\E{\end{equation}}
\def\BN{\begin{eqnarray*}}
\def\EN{\end{eqnarray*}}

\usepackage{cite}

\makeatletter
\@namedef{subjclassname@2020}{%
  \textup{2020} Mathematics Subject Classification}
\makeatother

\begin{document}
\title[Boltzmann equation with mixed boundary condition]{Boltzmann equation with mixed boundary condition}

\author[H.-X. Chen]{Hongxu Chen}
\address[HXC]{Department of Mathematics, The Chinese University of Hong Kong, Shatin, N.T., Hong Kong}
\email{hongxuchen.math@gmail.com}

\author[R.-J. Duan]{Renjun Duan}
\address[RJD]{Department of Mathematics, The Chinese University of Hong Kong, Shatin, N.T., Hong Kong.}
\email{rjduan@math.cuhk.edu.hk}
 
\date{\today}



\begin{abstract} 
We study the Boltzmann equation in a smooth bounded domain featuring a mixed boundary condition. Specifically, gas particles experience specular reflection in two parallel plates, while diffusive reflection occurs in the remaining portion between these two specular regions. The boundary is assumed to be motionless and isothermal. Our main focus is on constructing global-in-time small-amplitude solutions around global Maxwellians for the corresponding initial-boundary value problem. The proof relies on the $L^2$ hypocoercivity at the linear level, utilizing the weak formulation and various functional inequalities on the test functions, such as Poincar\'e and Korn inequalities. It also extends to the linear problem involving Maxwell boundary conditions, where the accommodation coefficient can be a piecewise constant function on the boundary, allowing for more general bounded domains. Moreover, we develop a delicate application of the $L^2-L^\infty$ bootstrap argument, which relies on the specific geometry of our domains, to effectively handle this mixed-type boundary condition. 
\end{abstract}

\maketitle

\section{Introduction}
In this paper we consider the initial-boundary value problem on the Boltzmann equation for rarefied gas contained in a smooth bounded domain $\Omega$ in $\R^3$:
\Be\label{1.1}
\p_t F + v\cdot \nabla_x F =  Q(F,F),
\Ee 
where $F=F(t,x,v)\geq 0$ stands for the velocity distribution function of gas particles with velocity $v=(v_1,v_2,v_3)\in \R^3$ at time $t\geq 0$ and position $x=(x_1,x_2,x_3)\in \Omega\subset \R^3$, and the initial and boundary conditions are to be specified later. The Boltzmann collision term is a bilinear integral operator acting only on velocity variable and for the hard sphere model it reads as 
\Be\notag
Q (G,F)(v):=\int_{\mathbb{R}^3}\int_{\mathbb{S}^2}
|(v-u) \cdot \o|
     \big[G(u')F(v')-G(u)F(v) \big]\dd \omega \dd u.
\Ee
with
\begin{equation*}
v'=v-(v-u)\cdot \omega\omega,\quad u'=v+(v-u)\cdot\omega\omega,
\end{equation*}
which satisfies the conservation laws of momentum and energy for intermolecular elastic collisions:
\begin{equation*}
v'+u'=v+u,\quad |v'|^2+|u'|^2=|v|^2+|u|^2.
\end{equation*}
Through the paper we consider only the hard sphere collision kernel $|(v-u)\cdot \omega|$ but other kinds of collision kernels under the Grad's angular cutoff assumptions could also be studied in a similar way, cf.~\cite{Duan}. 

When the rarefied gas is confined in a container, the interaction between the gas and the boundary plays a key role. To the end, we assume that the boundary is motionless. To describe the boundary condition, we denote the boundary of the phase space as
\begin{equation*}
\gamma:= \{(x,v)\in \p\O \times \mathbb{R}^3\}.
\end{equation*}
Let $n=n(x)$ be the outward normal direction at $x\in\p\O$. We decompose $\gamma$ as
\begin{equation*}
\begin{split}
    &\gamma_- = \{(x,v)\in \p\O\times \mathbb{R}^3 : n(x)\cdot v < 0\},  \\
    & \gamma_+ = \{(x,v)\in \p\O\times \mathbb{R}^3 : n(x)\cdot v > 0\}, \\
    &\gamma_0 = \{(x,v)\in \p\O\times \mathbb{R}^3 : n(x)\cdot v = 0\}.
\end{split}    
\end{equation*}

The Boltzmann equation holds a crucial position within collisional kinetic theory, attracting a vast amount of literature dedicated to exploring its well-posedness theory with physical boundary conditions, cf.\cite{CerBook,mischler,asano,DV,guiraud,G2}. In a pioneering work, Guo \cite{G} developed an $L^2-L^\infty$ framework that established global well-posedness and exponential convergence to the global Maxwellian. These results encompass boundary conditions including inflow, diffuse reflection, specular reflection, and bounce-back reflection. When considering the non-isothermal boundary, \cite{EGKM,EGKM2} investigated a steady Boltzmann equation with dynamical stability, where a constructive method was introduced to establish $L^2$ hypocoercivity and $L^6$ control of the macroscopic quantities. It is worth noting that achieving the $L^\infty$ bound and well-posedness in the case of pure specular reflection requires more intricate assumptions and techniques. Specifically, the boundary is assumed to be strictly convex and analytic in \cite{G}, and in \cite{KL}, a triple iteration technique was developed to relax the analytic assumption.

When it comes to boundaries that exhibit both specular and diffuse reflection properties, the Maxwell boundary condition and the generalized diffuse boundary condition serve as suitable models for capturing these phenomena. These conditions represent a linear and nonlinear combination of specular and diffuse reflections, respectively, through the accommodation coefficients. In the case of the Maxwell boundary condition, global well-posedness for the Boltzmann equation is proven in \cite{briant} under certain constraints on the accommodation coefficient. 

For general smooth accommodation coefficients, a recent progress by Bernou et al \cite{Bernou} utilized the Korn's inequality from \cite{DV2} to construct a remarkable symmetric Poisson system (as described in Lemma \ref{lemma:sym_poisson}). This symmetric Poisson system plays a pivotal role in establishing the $L^2$ bound for the momentum component. Notably, the presence of the smooth coefficient $\alpha(x)$ in the boundary condition introduced in Lemma \ref{lemma:sym_poisson} provides the connection to the Maxwell boundary condition, leading to $L^2$ hypocoercivity of the linear problem. Such elliptic system also has been recently applied in \cite{chen2023} to construct an $L^6$ control of the macroscopic quantities with pure specular boundary condition.

Turning to the generalized diffuse boundary, also known as the Cercignani-Lampis boundary in \cite{C,CIP}, the first author of this paper \cite{chen} demonstrated local well-posedness with a general accommodation coefficient and global well-posedness, dynamical stability with a small fluctuation assumption regarding the accommodation coefficient and wall temperature. In the context of the transport equation and the linear Boltzmann equation, a wider range of accommodation coefficients are allowed to achieve global convergence solutions, as discussed in \cite{bernou2022, bernou2023}.

The extensive study outlined above encompasses various scenarios, including well-posedness results for boundary conditions that lie between pure specular and diffuse reflection. However, it does not address the case where the boundary condition is a mixture of these two types, in other words, a portion of the boundary follows pure specular reflection while another portion follows diffuse reflection. Specifically ones may ask about the well-posedness for the initial-boundary value problem on the Boltzmann equation in bounded domains with a general Maxwell boundary condition such that the accommodation coefficient can be a piecewise constant function on the boundary. The investigation of well-posedness under such boundary conditions has remained an open question. In this paper, we aim to tackle this problem for a quite special shape of the boundary.

To the end we consider the following mixed boundary where gas particles are specularly reflected in two parallel plates and diffusively reflected in the rest portion between those two specular parts. Without loss of generality, we assume that two parallel plates are respectively located at $x_1=\pm L$ and the diffuse portion lies in $-L< x_1< L$. Precisely, for $x\in \partial\Omega$ with $-L< x_1< L$, we impose the diffuse boundary condition with constant temperature:
\begin{align}
  F(x,v)|_{\gamma_-:-L<x_1<L}  &  = \sqrt{2\pi}\mu(v) \int_{n(x)\cdot u>0} F(x,u) (n(x)\cdot u) \dd u,  \ -L<x_1<L,  
  \label{diffuse}
\end{align}
where $\mu$ corresponds to the normalized global Maxwellian:
\begin{align}\label{Maxwellian}
\mu(v):= \frac{1}{(2\pi)^{3/2}} e^{-\frac{|v|^2}{2}}.
\end{align}
For $x\in \partial\Omega$ with $x_1=-L$ and $x_1=L$, we impose the specular boundary condition:
\begin{align}
  F(x,v_1,v_2,v_3)|_{\gamma_-: x_1 \in \{-L,L\}}  & = F(x,-v_1,v_2,v_3), \ x_1=-L, L.  
  \label{specular}
\end{align}
For brevity we also denote the diffuse boundary portion and the specular boundary portion respectively as
\begin{equation}\label{bdr_portion}
\p \O_1 = \{x\in \p\O: x_1 \in (-L,L)\}, \ \ \p \O_2 = \{x \in \p\O: x_1 \in \{-L,L\}\}.
\end{equation}
Then the boundary is $\p\O = \p\O_1 \cup \p\O_2$. With the special boundary shape, we assume the domain is $C^1$.

In the standard perturbation framework, we seek for the solution of the form $F = \mu + \sqrt{\mu}f$ for the problem \eqref{1.1}, \eqref{diffuse} and \eqref{specular}. Then, we may reformulate it as 
\begin{equation}
\label{f_eqn}
\left\{\begin{aligned}
  &\p_t f + v\cdot \nabla_x f +  \mathcal{L} f  =  \Gamma(f,f) \quad \text{in }(0,\infty)\times \Omega\times \R^3,\\
   &            f(t,x,v)|_{\gamma_-}       = \sqrt{2\pi}\sqrt{\mu(v)} \int_{n(x)\cdot u>0} f(t,x,u)\sqrt{\mu(u)} |n(x)\cdot u| \dd u \quad \text{for }x\in \p\O_1, \\ 
  &          f(t,x,v)|_{\gamma_-}    = f(t,x,-v_1,v_2,v_3) \quad \text{for }x\in \p\O_2,
\end{aligned}\right.
\end{equation}
where the linearized collision operator $\mathcal{L}$ and nonlinear collision operator $\Gamma(\cdot,\cdot)$ are respectively given by
\begin{align}\label{def.L}
\mathcal{L}f=-\mu^{-1/2}\left\{Q(\mu,\sqrt{\mu}f)+Q(\sqrt{\mu}f,\mu)\right\},
\end{align}
and
\begin{align}\label{def.Ga}
\Gamma (g,f)=\mu^{-1/2} Q(\sqrt{\mu}g,\sqrt{\mu}f).
\end{align}

We denote a velocity weight as 
\begin{equation*}
w(v):= e^{\theta |v|^2}, \ \ 0<\theta < \frac{1}{4}.    
\end{equation*}
The main result of this paper is stated as follows.

\begin{theorem}\label{thm:linfty}
Assume $\O$ is bounded and smooth, where the boundary is given by \eqref{bdr_portion}. There is a constant $\delta>0$ such that if $F_0(x,v):=\mu+\sqrt{\mu}f_0(x,v)\geq 0$ satisfies $\int_{\Omega}\int_{\R^3}\sqrt{\mu}f_0(x,v)\,\dd v\dd x=0$ and
\begin{equation}\label{initial_condition}
\Vert wf_0\Vert_\infty < \delta,
\end{equation}
then there exists a unique solution $f(t,x,v)$ to the problem \eqref{f_eqn} with the initial condition $f(0,x,v)=f_0(x,v)$ such that $F(t,x,v):=\mu+\sqrt{\mu}f(t,x,v)\geq 0$, $\int_{\Omega}\int_{\R^3}\sqrt{\mu}f(t,x,v)\,\dd v\dd x\equiv 0$, and also the following estimate holds true:
\begin{equation*}
\Vert wf(t)\Vert_\infty \leq 2Ce^{-\lambda t}\delta,
\end{equation*}
for all $t\geq 0$, where $C$ and $\lambda>0$ are constants independent of $t$.

\end{theorem}

We would like to emphasize the importance of the special boundary conditions \eqref{diffuse} and \eqref{specular} in obtaining the $L^\infty$ estimate in Theorem \ref{thm:linfty}. In a general smooth bounded domain with specular reflection, the trajectory becomes intricate, making it necessary to employ a triple iteration technique in \cite{KL} to perform the change of variable within the $L^2-L^\infty$ framework developed in \cite{G}. In our specific boundary, where the specular portion is only located at $x_1 \in \{-L, L\}$, the trajectory between specular reflections becomes simplified. 

It is also important to note that the requirement for this specific boundary arises solely from the need to establish the $L^\infty$ estimate. In fact, the convexity assumption is not necessary, and for the $L^2$ estimate, we do not rely on such a choice of boundary. In other words, our proof for the $L^2$ hypocoercivity in Proposition \ref{prop:l2} remains valid for a general smooth domain without any restrictions on the position of the specular portion.

The proof of the $L^2$ hypocoercivity in Proposition \ref{prop:l2} relies on choosing special test functions, which are constructed using the elliptic systems in Lemma \ref{lemma:poisson} and Lemma \ref{lemma:sym_poisson}. It is worth noting that both of these systems require smooth accommodation coefficients to achieve $H_x^2$ regularity. However, the accommodation coefficients in the boundary condition in Proposition \ref{prop:l2} is a piecewise and discontinuous function that only takes values in ${0,1}$. To overcome this difficulty, a key observation is that for the discontinuous function $\alpha(x)\in \{0,1\}$, it is possible to construct a smooth version of $\alpha(x)$ denoted as $\beta(x)$. This new function $\beta(x)$ satisfies the property $\beta(x)\leq \alpha(x)$ and is not identical to zero. By employing $\beta(x)$ in the elliptic systems of Lemma \ref{lemma:poisson} and Lemma \ref{lemma:sym_poisson}, we can ensure the smoothness of the solutions and establish $H^2_x$ regularity. Furthermore, the choice of $\beta(x)$ ensures that whenever $\alpha(x)=0$ (corresponding to the specular portion), we also have $\beta(x)=0$. This allows that the contribution of the specular portion vanishes in the boundary integral of the weak formulation, as expressed in \eqref{weak_formula}.

The proof of the $L^\infty$ estimate requires delicate analysis of the characteristics. In Definition \ref{def:sto_cycle_spec} and Definition \ref{def:sto_cycle_diffuse}, we introduce the concept of stochastic cycles for specular reflection and diffuse reflection, respectively. We consider specular reflection as an intermediate process between two diffuse reflections. It is well-known that the occurrence of a large number of interactions with diffuse reflection is of low probability (see \cite{G} and Lemma \ref{lemma:tk}). Consequently, we obtain a finite-fold integration characteristic formula, as shown in \eqref{initial} to \eqref{bdr_g_i}. The main difficulty lies in estimating the stochastic cycles involving specular reflection. As mentioned before, the specific boundary simplifies the trajectory. Furthermore, we can treat the specular reflection at $x_1 \in (-L, L)$ as an extension to $x_1 \in (-NL, NL)$ when the velocity of the characteristic is bounded. This allows us to convert the velocity integral into a spatial integral and perform a change of variables. In the $L^2-L^\infty$ framework, the Jacobian of this change of variable under the simplified trajectory is easy to compute. Further details can be found in the proof of Proposition \ref{prop:linfty}.


The analysis presented in this work provides insights to the future study of the Maxwell boundary conditions for general bounded domains, considering a general accommodation coefficient that can be a piecewise constant function on the boundary .

The outline of this paper is given as follows. In Section \ref{sec:prelim}, we list several properties of the linear and nonlinear collision operator. In Section \ref{sec:l2}, we employ the test function method with Korn's inequality and Poincar\'e inequality to establish the crucial $L^2$ hypocoercivity. Finally, in Section \ref{sec:linfty}, we utilize a boot-strap argument to establish the $L^\infty$ estimate, thereby concluding Theorem \ref{thm:linfty}.

\medskip
\noindent\textbf{Notation.} Throughout the paper, we adapt the following notations:
\begin{align*}
 \Vert f\Vert_\infty   & = \Vert f(x,v)\Vert_{L_{x,v}^\infty(\O\times \mathbb{R}^3)}, \ \ \
|f|_\infty         = \Vert f(x,v)\Vert_{L^\infty_{x,v}(\p\O\times \mathbb{R}^3)} ,\\
   \Vert f\Vert_{L^2}       &= \Vert f\Vert_{L^2_{x,v}(\O\times \mathbb{R}^3)}, \ \ \ \Vert f\Vert_{L^2_x} = \Vert f\Vert_{L^2_x(\O)}, \ \ \ \Vert f\Vert_{L^2_\nu} = \Vert \nu(v) f \Vert_{L^2_{x,v}(\O\times \mathbb{R}^3)}, \\
  \int_\gamma f(x,v)\dd \gamma &  = \int_{\gamma} f(x,v)(n(x)\cdot v) \,\dd S_x \dd v ,\ \ \ 
|f|_{2,+}^2 = \int_{\gamma_+} |f(x,v)|^2 (n(x)\cdot v) \,\dd S_x \dd v .
\end{align*}
Here $\dd S_x$ corresponds to the surface integral on the boundary. 

\ \\

\section{Preliminary}\label{sec:prelim}

Recall the linearized Boltzmann operator $\mathcal{L}$ as in \eqref{def.L}. One can write 
\begin{equation}
\label{Lsplit}
\mathcal{L} f = \nu(v) f- Kf,
\end{equation}
where $\nu(v)$ is a velocity multiplication and $K$ is an integral operator.  
First we have the classical Grad estimate.

\begin{lemma}[\cite{R}]
The linearized Boltzmann operator $K$ in~\eqref{Lsplit} is given by
\[
Kf(x,v)=\int_{\mathbb{R}^3}\mathbf{k}(v,u)f(x,u)\,\dd u.
\]

The kernel $\mathbf{k}(v,u)$ satisfies:
\Be\notag
 |\mathbf{k}  (v,u)| \lesssim \mathbf{k}_\varrho (v,u), \   \ \mathbf{k}_\varrho (v,u) := e^{- \varrho |v-u|^2}/ |v-u|.
  \Ee

The collision frequency $\nu(v)$ in \eqref{Lsplit} satisfies
\begin{equation}\label{nu_bdd}
\nu(v) \sim\sqrt{|v|^2+1}.
\end{equation}

\end{lemma}

\begin{lemma}\label{lemma:k_theta}
Let $0\leq \theta < \frac{1}{4}$, denote $\mathbf{k}_\theta(v,u) := \mathbf{k}(v,u) \frac{e^{\theta |v|^2}}{e^{\theta |u|^2}}.$ There exists $C_\theta > 0$ such that
\begin{equation}\label{k_theta}
\int_{\mathbb{R}^3}  \mathbf{k}(v,u) \frac{e^{\theta |v|^2}}{e^{\theta |u|^2}}  \,\dd u  \leq \frac{C_\theta}{1+|v|}.
\end{equation}
Moreover, for $N\gg 1$, we have
\begin{equation}\label{k_N_upper_bdd}
\mathbf{k}_\theta(v,u) \mathbf{1}_{|v-u|> \frac{1}{N}} \leq C_N,
\end{equation}
and
\begin{equation}\label{K_N_small}
\int_{|u|>N \text{ or } |v-u|\leq \frac{1}{N}} \mathbf{k}_\theta(v,u) \,\dd u \lesssim \frac{1}{N} \leq o(1).
\end{equation}

\end{lemma}

\begin{proof}
The proof mostly follows from Lemma 3 in \cite{G}, where for $0\leq \theta < \frac{1}{4}$, we can find $\e = \e(\theta)$ such that
\begin{align}
    \mathbf{k}_\theta(v,u) \leq \big[\frac{1}{|v-u|}+ |v-u| \big]e^{-\e \big[|v-u|^2 + |v\cdot (v-u)| \big]}. \label{k_theta_bdd}
\end{align}    
Thus \eqref{k_theta} follows by the factor $e^{-\e|v\cdot (v-u)|}$.

Clearly, with the exponential decay in $|v-u|$, we conclude \eqref{k_N_upper_bdd}.

For \eqref{K_N_small}, directly applying \eqref{k_theta_bdd} we have
\begin{align*}
    &\int_{|v-u|\leq \frac{1}{N}} \mathbf{k}_\theta(v,u) \,\dd u \lesssim o(1)   .
\end{align*}
When $|u|>N$, we split the cases into $|v|>\frac{N}{2}$ and $|v|\leq \frac{N}{2}$. In the first case, \eqref{K_N_small} follows by applying \eqref{k_theta}. For the other case, we have $|v-u|>\frac{N}{2}$, then \eqref{K_N_small} follows from \eqref{k_theta_bdd}.
\end{proof}

Regarding the nonlinear collision operator $\Gamma(\cdot,\cdot)$ as in \eqref{def.Ga}, it is direct to derive the following estimate.

\begin{lemma}[Lemma 5 of \cite{G}]\label{lemma:gamma}
For $0\leq \theta < \frac{1}{4}$, we have
\begin{equation*}
 \Big\Vert \frac{e^{\theta |v|^2}}{\langle v\rangle} \Gamma(f,g)\Big\Vert_\infty \lesssim \Vert wf\Vert_\infty \Vert wg\Vert_\infty.
\end{equation*}

\end{lemma}

\ \\

\section{$L^2$ estimate}\label{sec:l2}

It is well-known that the linearized collision operator $\mathcal{L}$ as in \eqref{def.L} has a 5 dimensional null space $\mathcal{N}$ spanned by the orthonormal basis 
\begin{align*}
    & \chi_0(v): = \sqrt{\mu(v)}, \ \ \chi_i(v): = v_i \sqrt{\mu(v)}, \ i=1,2,3, \ \ \chi_4(v): = \frac{|v|^2-3}{\sqrt{6}}\sqrt{\mu(v)}.
\end{align*}
We denote $\mathbf{P}f$ as the $L^2_v$ projection of $f$ onto $\mathcal{N}$:
\begin{align*}
   \mathbf{P}f & := \sum_{i=0}^4 \langle f,\chi_i\rangle \chi_i =  a(x)\chi_0 + \sum_{i=1}^3 b_i(x)\chi_i + c(x) \chi_4,
\end{align*}
with
\begin{equation*}
   a(t,x)   : = \langle f,\chi_0\rangle ,  \ \ \mathbf{b}(t,x)=(b_1(t,x),b_2(t,x),b_3(t,x)), \ 
   b_i(t,x)   := \langle f,\chi_i\rangle
   \ \text{ for }
   i=1,2,3 ; \ \
   c(t,x)   := \langle f,\chi_4 \rangle, 
\end{equation*}
where we have taken the usual inner product on $L^2(\R^3_v)$: 
\begin{equation}
\label{velocity_inner}
 \langle f,g\rangle     = \int_{\mathbb{R}^3} f(v)g(v)\,\dd v.
\end{equation}

In this section we consider the solution to the linearized Boltzmann equation
\begin{align}
   \p_t f + v\cdot \nabla_x f + \mathcal{L}f & =   g,    
   \label{linear_f}
\end{align}
with an inhomogeneous source term $g=g(t,x,v)$ satisfying 
\begin{equation}\label{assumption_g}
\mathbf{P}g=0.   
\end{equation}

As mentioned before, the special boundary conditions are proposed for analyzing the characteristic for the $L^\infty$ estimate. For $L^2$ estimate of the linearized equation \eqref{linear_f}, we do not need such restriction. In the following proposition we give a decay-in-time $L^2$ estimate for general boundaries.

We consider the Maxwell boundary condition
\begin{equation}\label{maxwell}
f(t,x,v)|_{\gamma_-} = \alpha(x) \sqrt{2\pi}\mu(v) \int_{n(x)\cdot u>0} f(t,x,u)\sqrt{\mu(u)} (n(x)\cdot u) \,\dd u + (1-\alpha(x)) f(x,\mathfrak{R}_xv),
\end{equation}
where $\mathfrak{R}_x v = v-2(n(x)\cdot v)n(x)$ and $\alpha(x)\in [0,1]$ with $x\in \partial\Omega$ is the accommodation coefficient on the boundary.

Through the section we allow the accommodation coefficient to be discontinuous, subject to the following assumption:

\begin{assumption}\label{assumption:coeff}
For the accommodation coefficient, it holds that $\alpha(x)\in \{0,1\}$ for $x\in \p\O$, namely, $\alpha(x)$ takes only values either $0$ or $1$ on $\partial\Omega$. Moreover, there exists a smooth function $\beta(x)\in [0,1]$ on $\partial\Omega$ such that $\beta(x)\leq \alpha(x)$ for any $x\in \p\Omega$ and $\beta(x)\not\equiv 0$ on $\partial\Omega$.
\end{assumption}

With the above assumption, we have the following result.

\begin{proposition}\label{prop:l2}
Let $\O$ be an arbitrary bounded and $C^1$ domain. Let \eqref{assumption_g} hold for all $t>0$ and the accommodation coefficients satisfy Assumption \ref{assumption:coeff}, then
there exists $0<\lambda\ll 1$ such that if the initial data $f_0$ and source data $g$ satisfy
\[\Vert f_0\Vert_{L^2}^2 + \int_0^t \Vert e^{\lambda s}g(s)\Vert_{L^2}^2\,\dd s < \infty,\]
then there exists a unique solution to the problem
\begin{equation}\label{intial_value_problem}
\p_t f + v\cdot \nabla_x f + \mathcal{L}f =   g, \ f(0,x,v) = f_0(x,v)
\end{equation}
with $f$ satisfying the boundary condition in~\eqref{maxwell}. 
Moreover, we have that
\begin{equation}\label{l2_decay}
\Vert f(t)\Vert_{L^2}^2 \lesssim e^{-\lambda t} \Big\{\Vert f_0\Vert_{L^2}^2 + \int_0^t  \Vert e^{\lambda s}g(s)\Vert_{L^2}^2 \,\dd s  \Big\},\quad\forall\,t\geq 0.
\end{equation}

\end{proposition}

\begin{remark}
The boundary condition in \eqref{diffuse} and \eqref{specular} corresponds to a special case of \eqref{maxwell}, where the accommodation coefficients $\alpha(x)=0$ when $x_1=-L,L$, and $\alpha(x)=1$ when $-L<x_1<L$. Clearly such $\alpha(x)$ satisfies Assumption \ref{assumption:coeff} by choosing a smooth function $\beta(x)$ such that
\begin{equation*}
    \beta(x) = \begin{cases}
        & 0 \text{ for } x_1\in \{-L,L\},\\
        & 1 \text{ for } x_1 \in [-L/2,L/2],\\
        & \in (0,1) \text{ for } x_1 \in (-L,-L/2) \cup (L/2,L).
    \end{cases}
\end{equation*}
\end{remark}

From the well-known Weyl's theorem, we have $\langle \mathcal{L}f,f\rangle \gtrsim \Vert (\mathbf{I}-\mathbf{P})f\Vert_{\nu}^2$. To prove the proposition above, we need to have the following $L^2$ control of the macroscopic quantities.

\begin{lemma}\label{lemma:L2}
Suppose $f$ solves the following equation with $0\leq \lambda\ll 1$,
\begin{equation}\label{linear_f_lambda}
\p_t f +v\cdot \nabla_x f + \mathcal{L}f  =  \lambda f + g,
\end{equation}
with boundary condition \eqref{maxwell} and accommodation coefficients $\alpha(x)$ satisfying Assumption \ref{assumption:coeff}. If $g$ satisfies \eqref{assumption_g}, then we have
\begin{align*}
   \int_s^t \Vert \mathbf{P}f(\tau)\Vert_{L^2}^2 \,\dd \tau & \lesssim G(t)-G(s) + \int_s^t \Vert (\mathbf{I}-\mathbf{P})f\Vert_{L^2}^2 \,\dd \tau + \int_s^t \Vert g(\tau)\Vert_{L^2}^2 \,\dd \tau\notag \\
   &  \ \ \ + \int_s^t  |\alpha(x)(1-P_\gamma) f(\tau)|_{2,+}^2 \,\dd \tau,
\end{align*}
where $G(t)$ is a functional of $f(t,x,v)$ such that $|G(t)|\lesssim \Vert f(t)\Vert_{L^2}^2$ holds true for any $t\geq 0$.

\end{lemma}

\begin{remark}
We note that this lemma covers the case of $\lambda =0$, where \eqref{linear_f_lambda} becomes the same as \eqref{linear_f}. Here we include the term $\lambda f$ on the RHS of \eqref{linear_f_lambda} in order to study the decay-in-time $L^2$ estimate, where the equation of $e^{\lambda t}f$ will induce an extra term $\lambda e^{\lambda t}f$.
\end{remark}

To prove this lemma, we first cite some preliminary results. Denote $V:=H^1(\O)$. 

\begin{lemma}[\cite{Bernou}]\label{lemma:poisson}
Let the accommodation coefficient $\alpha(x)\in [0,1]$ be a Lipshitz function on $\p\O$ and $\alpha(x)$ be not identical to $0$ on $\partial\Omega$. 
\begin{itemize}
\item[(i)] The following Poincar\'e inequality holds true:
\begin{align}
   \Vert u\Vert_{L^2_x}^2 & \lesssim \Vert \nabla u \Vert_{L^2_x}^2 + \int_{\p\O} \frac{\alpha(x)}{2-\alpha(x)} u^2 \,\dd S_x. \label{poincare}
\end{align}

  \item[(ii)] Consider the following Poisson equation for a scalar function $u=u(x)$:
\begin{align*}
   -\Delta u & = h \text{ in }\O\notag\\
           (2-\alpha(x))\nabla u \cdot n + \alpha(x)u & = 0  \text{ on }\p\O,
\end{align*}
with $h\in L^2(\O)_x$. Then, there exists a unique solution $u\in V$ satisfying the following weak formulation:
\begin{align*}
    \int_\O \nabla u \cdot \nabla v  \,\dd x + \int_{\p\O} \frac{\alpha(x)}{2-\alpha(x)} u v \,\dd S_x = \int_\O h v \,\dd x  \text{ for all } v\in V,
\end{align*}
and moreover, we have $u\in H^2_x(\O)$  which satisfies the estimate
\begin{align}
  \Vert u\Vert_{H^2_x}  & \lesssim \Vert h\Vert_{L^2_x}. \label{poisson_h2}
\end{align}

\end{itemize}

\end{lemma}

For a vector-valued function $\mathbf{u}=\mathbf{u}(x) = (u_1(x),u_2(x),u_3(x))\in \R^3$, we denote the Jacobian matrix and its symmetric part (a.k.a. the deformation tensor) respectively as
\Be\notag
(\nabla \mathbf{u})_{ij} = \frac{\p u_i}{\p x_j}
, \ \  \ \ 
(\nablaS \mathbf{u})_{ij} = \frac{1}{2} \left(
\frac{\p u_i}{\p x_j} + \frac{\p u_j}{\p x_i}
\right).
\Ee
The antisymmetric part of the Jacobian matrix is denoted by 
\Be\notag
(\nablaA \mathbf{u})_{ij} = (\nabla \mathbf{u})_{ij} - (\nablaS \mathbf{u})_{ij}: = \frac{1}{2} \left(
\frac{\p u_i}{\p x_j} - \frac{\p u_j}{\p x_i}
\right).
\Ee
We have the following identity:
\Be\label{div_Delta}
 \Div(\nablaS \mathbf{u}) = {\color{black}\frac{1}{2}} \left(\Delta \mathbf{u} + \nabla \Div  \mathbf{u}\right).
 \Ee

Define a Hilbert space 
\begin{equation}\label{hilbert_X}
\mathcal{X} = \{\mathbf{u}\in H^1_x(\O)|\mathbf{u}(x)\cdot n(x) = 0 \text{ on }\p\O\}   . 
\end{equation}

\begin{lemma}[\cite{Bernou} and \cite{DV2}]\label{lemma:sym_poisson}
Consider the same accommodation coefficient $\alpha(x)$ as in Lemma \ref{lemma:poisson}. 
\begin{itemize}
  \item[(i)] The following Korn's inequality holds: for any $\mathbf{u}(x)\in H^1_x(\O)$ such that $\mathbf{u}(x)\cdot n(x)=0$ on $\p\O$, one has
\begin{equation}\label{korn}
\Vert \mathbf{u}\Vert_{H^1_x}^2 \lesssim \Vert \nablaS \mathbf{u}\Vert_{L^2_x}^2 +  \int_{\p\O} \frac{\alpha(x)}{2-\alpha(x)} |\mathbf{u}|^2 \,\dd S_x.
\end{equation}
\item[(ii)] Consider the following symmetric Poisson system
\begin{equation*}
\left\{\begin{aligned}
  \Div(\nablaS \mathbf{u})    &= h  \text{ in }\O \\
   \mathbf{u}\cdot n  &=   0 \text{ on }\p\O \\
   (2-\alpha(x)) [\nablaS \mathbf{u} n - (\nablaS \mathbf{u} : n\otimes n)n ] + \alpha(x) \mathbf{u}  &= 0 \text{ on } \p\O , 
\end{aligned}\right.
\end{equation*}
with $h\in L^2_x(\O) $. Then, there exists a unique solution $\mathbf{u}\in \mathcal{X}$ satisfying the following weak formulation:
\begin{align}
    & \int_\O \nablaS \mathbf{u} : \nablaS \mathbf{v} \dd x + \int_{\p\O} \frac{\alpha(x)}{2-\alpha(x)} \mathbf{u} \cdot \mathbf{v} \dd S_x = \int_{\O} \mathbf{v}\cdot h  \dd x \ \ \text{ for all }\mathbf{v}\in \mathcal{X},
    \label{sym_poisson_weak}
\end{align}
and moreover, we have $\mathbf{u}\in H^2_x(\O)$ which satisfies the estimate
\begin{align}
    & \Vert \mathbf{u}\Vert_{H^2_x}\lesssim \Vert h\Vert_{L^2_x}.  \label{sym_poisson_h2}
\end{align}
\end{itemize}

\end{lemma}

\begin{remark}
The Korn's inequality was investigated by Desvillettes and Villani \cite{DV2} under the condition that $\mathbf{u}(x)\cdot n(x)=0$ on $\partial\Omega$. The inequality can be stated as follows:
\[\inf_{R\in \mathcal{R}_\Omega} \Vert \nabla (\mathbf{u}-R)\Vert_{L^2_x}^2 \lesssim \Vert\nablaS \mathbf{u}\Vert_{L^2_x}^2,\]
where $\mathcal{R}_\Omega$ represents the space of centered infinitesimal rigid displacement fields, as defined in \cite{DV2}. Such Korn's inequality was applied in \cite{Bernou} to study the elliptic system \eqref{sym_poisson_weak} for $\alpha(x) \equiv 0$ with taking both axisymmetric and non-axisymmetric domain into consideration. The construction of the elliptic system further leads to the development of the $L^2$ hypocoercivity with pure specular reflection boundary in \cite{Bernou} and \cite{chen2023}.

The Korn's inequality in \eqref{korn} has been established in \cite{Bernou}, which leads to the development of the $L^2$ hypocoercivity with Maxwell boundary condition with smooth accommodation coefficients.
\end{remark}

\begin{proof}[\textbf{Proof of Lemma \ref{lemma:L2}}]
For consistency of the notation, in the proof we still denote $\p \O_1$ and $\p \O_2$ as the diffuse boundary portion and specular boundary portion respectively:
\begin{equation}\label{portion}
\p\O_1 = \{x\in \p\O: \alpha(x)=1\}, \ \ \p\O_2 = \{x\in \p\O: \alpha(x)=0\}.
\end{equation}

We will heavily use the following change of variable for the diffuse reflection and specular reflection respectively.
For the diffuse reflection
\begin{align*}
    & \int_{n(x)\cdot v<0} f(v) (n(x)\cdot v) \dd v,
\end{align*}
let $u = v-2(n(x)\cdot v)n(x)$, then the Jacobian is $1$ and $n(x)\cdot u = -n(x)\cdot v$, thus $v = u+2(n(x)\cdot v)n(x) = u - 2(n(x)\cdot u)n(x)$. We derive that
\begin{align}
    & \int_{n(x)\cdot v<0} f(v) (n(x)\cdot v) \dd v = -\int_{n(x)\cdot u>0} f(u-2(n(x)\cdot u)n(x)) (n(x)\cdot u) \dd u \notag\\
    & = -\int_{n(x)\cdot v>0} f(v-2(n(x)\cdot v)n(x)) (n(x)\cdot v) \dd v.  \label{cov_diffuse}
\end{align}
Similarly, for the specular reflection,
\begin{align}
    & \int_{n(x)\cdot v<0} g(v)f(v-2(n(x)\cdot v)n(x)) (n(x)\cdot v) \dd v \notag \\
    &= -\int_{n(x)\cdot u>0} g(u-2(n(x)\cdot u)n(x))f(u) (n(x)\cdot u) \dd u \notag\\
    & = -\int_{n(x)\cdot v>0} g(v-2(n(x)\cdot v)n(x)) f(v) (n(x)\cdot v) \dd v.  \label{cov_spec}
\end{align}

We control the $L^2$ bound of the macroscopic quantities $a,\mathbf{b},c$ using special test functions with the weak formulation of~\eqref{linear_f}:
\begin{equation}\label{weak_formula} 
\begin{split}
 & \int_{\O\times \mathbb{R}^3} \{\psi f(t)-\psi f(s)\}\dd x \dd v - \int_s^{t}\int_{\O\times \mathbb{R}^3} f \p_\tau \psi \dd x \dd v \dd \tau \\
    & = \int_s^{t}\int_{\O\times \mathbb{R}^3} v \cdot \nabla_x \psi f\dd x \dd v \dd \tau - \int_s^{t} \int_{\gamma} \psi f \dd \gamma \dd \tau \\
    & - \int_s^t \int_{\O\times \mathbb{R}^3} (\mathcal{L} f)\psi \dd v \dd x \dd \tau + \int_s^t \int_{\O\times \mathbb{R}^3} (\lambda f + g)\psi \dd x \dd v \dd \tau \\
    & := I_1+I_2+I_3+I_4.  
\end{split}
\end{equation}

We denote the Burnette function of the space $\mathcal{N}^\perp$:
\begin{equation}\label{A_j}
\begin{split}
    \hat A_{ij}(v) &:=  \left(v_i v_j - \frac{\delta_{ij}}{3}|v|^2\right) \sqrt \mu  \text{ for } i,j=1,2,3,   \\
 \hat{B}_{j}   & :=  v_i \frac{|v|^2-5}{\sqrt{10}} \mu^{1/2} \text{ for } i = 1,2,3.
\end{split}
\end{equation}
A key property for these functions is
\begin{equation}\label{A_B_proj_0}
\mathbf{P}(\hat{A}_{ij}) = \mathbf{P}(\hat{B}_j)=0.
\end{equation}

By Assumption \ref{assumption:coeff}, we can choose a smooth $\beta(x)$ such that
\begin{equation}\label{beta_property}
\beta(x)\not \equiv 0, \ \ \beta(x) \leq \alpha(x), \ \ \Vert \beta\Vert_{C^1(\p\O)} <\infty.
\end{equation}

We will use the following basic facts for the velocity integral:
\begin{align*}
  \int_{-\infty}^{\infty} e^{-v^2/2} \dd v    = \sqrt{2\pi}, \ \ \int_{-\infty}^{\infty} v^2 e^{-v^2/2} \dd v  = \sqrt{2\pi} , \ \  \int_{-\infty}^{\infty} v^4 e^{-v^2/2} \dd v  = 3\sqrt{2\pi}.
\end{align*}
The above computation leads to the following facts for the integral of $\mu(v)$ in~\eqref{Maxwellian}: for $i=1,2,3$
\begin{equation}\label{fact}
\begin{split}
     \int_{\mathbb{R}^3}   \mu(v) \dd v & = 1 , \ \  \int_{\mathbb{R}^3}  |v_i|^2  \mu(v) \dd v  = 1, \\
    \int_{\mathbb{R}^3} |v_i|^2 |v_j|^2 \mu(v)\dd v & =1 \ \text{ for } i\neq j,  \ \  \int_{\mathbb{R}^3} |v_i|^4 \mu(v) \dd v  = 3. 
\end{split}
\end{equation}

\textbf{Step 1: estimate of $c(x)$.}

We choose $\psi = \phi(x)\chi_4$ in~\eqref{weak_formula}. Then the transport operator on $\psi$ becomes
\begin{equation}\label{transport_c_1}
v\cdot \nabla_x \psi =  \sum_{i=1}^3 \frac{\sqrt{6}}{3} \p_i \phi \chi_i + \sum_{i=1}^3  \frac{\sqrt{15}}{3} \p_i \phi \hat{B}_i.
\end{equation}
Take $[s,t] = [t-\triangle,t]$, LHS of~\eqref{weak_formula} becomes
\[LHS = \int_{\O} \{ c(t) - c(t-\triangle)\} \phi \dd x.\]
By~\eqref{transport_c_1} we have
\begin{equation}\label{I1_c}
I_1 = \int^t_{t-\triangle} \int_\O \Big\{\frac{\sqrt{6}}{3}(\mathbf{b}\cdot \nabla_x \phi) + \sum_{i=1}^3  \frac{\sqrt{15}}{3} \p_i \phi \langle \hat{B}_i, (\mathbf{I}-\mathbf{P})f \rangle \Big\} \dd x \dd \tau,    
\end{equation}
while $I_3= 0$ from the property of $\mathcal{L}$, and the contribution of $g$ in $I_4$ is 0 from \eqref{assumption_g}. The contribution of $\lambda f$ in $I_4$ can be bounded as
\begin{align}
    & \lambda \int_{\O\times \mathbb{R}^3} f \psi \dd x \dd v = \lambda \int_{\O\times \mathbb{R}^3} f \phi(x) \chi_4 \dd x \dd v   = \lambda \int_{\O}  \phi(x) c(x)\dd x  \notag\\
    & \leq \lambda \Vert \phi\Vert_{L^2_x}^2 + \lambda \Vert c\Vert_{L^2_x}^2 .  \label{I4_c}
\end{align}

Then for fixed $t$ we choose $\phi = \Phi_c$ such that
\begin{align}
  -\Delta \Phi_c  & = \p_\tau c(t) \text{ in } \O \notag\\
  (2-\beta(x))\nabla \Phi_c \cdot n + \beta(x) \Phi_c & = 0 \text{ on }\p\O. \label{test_c_partial}
\end{align}
Here $\beta(x)$ is defined in \eqref{beta_property}.

The contribution of the boundary integral in $I_2$ of~\eqref{weak_formula} becomes
\begin{align}
     & \int_\gamma \Phi_c(x)\chi_4 f =  \int_{\p\O} \int_{\gamma_+} \Phi_c(x)\chi_4 f + \int_{\p\O} \int_{\gamma_-} \Phi_c(x)\chi_4 f \notag\\
    & = \int_{\p\O_1}\int_{\gamma_+} \Phi_c(x) \chi_4 f - \int_{\p\O_1}\int_{\gamma_+} \Phi_c(x)\chi_4 P_\gamma f . \label{I2_first_bdd}
\end{align}
Here we have applied a change of variable $v\to v-2(n(x)\cdot v)n(x)$ in \eqref{cov_diffuse}, and the contribution of $\p\O_2$ vanishes by applying the specular boundary condition in \eqref{cov_spec}. Then we proceed the computation as
\begin{align}
  \eqref{I2_first_bdd}  & \lesssim o(1)\Vert \Phi_c\Vert_{L^2(\p\O)}^2  + |\alpha(x)(1-P_\gamma)f|_{2,+}^2 . \label{I2_c}
\end{align}

By~\eqref{I1_c}, \eqref{I4_c} and \eqref{I2_c}, with $\lambda \ll 1$ and the trace theorem, we apply integration by part to have
\begin{align*}
& \int_{\O} |\nabla \Phi_c|^2  \dd x+   \int_{\p\O} \frac{\beta(x)}{2-\beta(x)} |\Phi_c|^2 \dd S_x  = -\int_{\O} \Delta \Phi_c \Phi_c \dd x = \int_{\O} \Phi_c \p_\tau c \dd x      \\
 & \lesssim \Vert \nabla \Phi_c\Vert_{L^2_x}\big[\Vert \mathbf{b}\Vert_{L^2_x} + \Vert \mu^{1/4}(\mathbf{I}-\mathbf{P})f\Vert_{L^2} \big] + o(1) \Vert \Phi_c \Vert_{H^1_x}^2 + o(1)\Vert c\Vert_{L^2_x}^2 + |\alpha(x)(1-P_\gamma)f|_{2,+}^2 .
\end{align*}
In the first equality, we applied the boundary condition in \eqref{test_c_partial}. Since $\beta$ satisfies the condition in Lemma \ref{lemma:poisson}, we apply \eqref{poincare} to conclude that
\begin{align*}
    &\Vert \Phi_c\Vert_{H^1_x}^2 \lesssim  \int_{\O} |\nabla \Phi_c|^2  \dd x+   \int_{\p\O} \frac{\beta(x)}{2-\beta(x)} |\Phi_c|^2 \dd S_x \\
    & \lesssim  \Vert \nabla\Phi_c\Vert_{L^2_x}\big[\Vert b\Vert_{L^2_x} + \Vert \mu^{1/4}(\mathbf{I}-\mathbf{P})f\Vert_{L^2} \big] + |\alpha(x)(1-P_\gamma)f|_{2,+}^2 + o(1)\Vert c\Vert_{L^2_x}^2 .
\end{align*}

Hence we obtain the estimate for $\Phi_c$ as
\begin{align}
   \Vert \Phi_c\Vert_{H^1_x} \lesssim o(1)\Vert c\Vert_{L^2_x}+\Vert \mathbf{b}\Vert_{L^2_x} + \Vert \mu^{1/4}(\mathbf{I}-\mathbf{P})f\Vert_{L^2} + |\alpha(x)(1-P_\gamma)f|_{2,+}. \label{test_c_partial_H1} 
\end{align}

Next we rearrange~\eqref{weak_formula} to have
\begin{equation}\label{weak_formula_2}
    \begin{split}
        & - \int_s^t \int_{\O\times \mathbb{R}^3} v\cdot \nabla_x \psi f \dd x \dd v \dd \tau   \\
    & = \int_{\O\times \mathbb{R}^3}\{-\psi f(t) + \psi f(s)\} \dd x \dd v +  \int_s^t \int_{\O\times \mathbb{R}^3} f \p_\tau \psi \dd x \dd v  \dd \tau  - \int_s^t \int_{\gamma} \psi f  \dd \gamma  \dd \tau   \\
    & -  \int_s^t  \int_{\O\times \mathbb{R}^3} \mathcal{L} f \psi  \dd x \dd v \dd \tau + \int_s^t \int_{\O\times \mathbb{R}^3} (\lambda f + g) \psi \dd  x \dd v \dd \tau  \\
    & := \{G_\psi(t) - G_\psi(s)\} + J_1 + J_2 + J_3 + J_4 .   
    \end{split}
\end{equation}
We let $\phi_c$ be a solution of the following problem
\begin{align}
  -\Delta \phi_c   & = c(t) \text{ in } \O  \notag \\ 
  (2-\beta(x))\nabla \phi_c \cdot n + \beta(x)\phi_c & = 0 \text{ on }\p\O, \label{poisson_c}
\end{align}
and we choose 
\begin{align}
  \psi  & :=\psi_c = \sum_{i=1}^3 \p_i \phi_c v_i(|v|^2-5) \mu^{1/2} =\sum_{i=1}^3 \sqrt{10}  \p_i \phi_c \hat{B}_i. \label{test_c}
\end{align}
A direct computation leads to
\begin{align*}
   -v\cdot \nabla_x \psi_c &  = -\frac{5\sqrt{6}}{3}\Delta \phi_c \chi_4  - \sum_{i,j=1}^3 \p_{ij}^2\phi_c (\mathbf{I}-\P)(v_iv_j (|v|^2-5)\mu^{1/2}).
\end{align*}
Thus the LHS of~\eqref{weak_formula_2} is
\begin{align}
  LHS  &  = \frac{5\sqrt{6}}{3} \int_s^t \int_{\O} c^2 \dd x\dd \tau - \sum_{i,j=1}^3 \int_s^t  \int_{\O} \p_{ij}^2 \phi_c  \langle (\mathbf{I}-\mathbf{P})f, v_iv_j (|v|^2-5)\mu^{1/2}\rangle \dd x \dd \tau \notag\\
  & = \frac{5\sqrt{6}}{3}\int_s^t \int_\O c^2 \dd x \dd \tau  + E_1, \label{LHS_c}
\end{align}
where, for any $\delta_1>0$, from~\eqref{poisson_h2},
\[|E_1|\lesssim \delta_1\int_s^t \Vert c\Vert_{L^2_x}^2\dd \tau + \frac{1}{ \delta_1}\int_s^t \Vert \mu^{1/4}(\mathbf{I}-\P)f\Vert_{L^2}^2 \dd \tau.\]
For $J_1$, using the fact that $\Phi_c = \p_\tau \phi_c $ and the estimate for $\Phi_c $ in~\eqref{test_c_partial_H1} we have
\begin{align}
  |J_1|  &  \lesssim  \delta_1\int_s^t \Vert \Phi_c \Vert^2_{H^1_x} \dd \tau  + \frac{1}{\delta_1} \int_s^t \Vert \mu^{1/4}(\mathbf{I}-\P)f\Vert^2_{L^2}  \dd \tau \notag\\
  & \lesssim \delta_1\int_s^t \Vert c\Vert_{L^2_x}^2 \dd \tau+ \delta_1 \int_s^t \Vert \mathbf{b}\Vert_{L^2_x}^2\dd \tau + \frac{1}{\delta_1} \int_s^t \Vert \mu^{1/4}(\mathbf{I}-\P)f\Vert^2_{L^2}  \dd \tau + \int_s^t|\alpha(x)(1-P_\gamma)f|_{2,+}^2 \dd \tau. \label{J1_c_bdd}
\end{align}
Then we apply boundary condition of $\phi_c$ and $f$ to compute $J_2$:
\begin{align*}
  \int_\gamma \psi f \dd \gamma  & = \int_{\gamma_+} \psi f \dd \gamma  +  \int_{\gamma_-} \psi f \dd \gamma.
\end{align*}
First we compute the specular part, we have
\begin{align}
 & \int_{\p\O_2} \bigg[ \int_{n(x)\cdot v>0} + \int_{n(x)\cdot v<0}\bigg] \big(|v|^2 -5\big)\sqrt{\mu} (v\cdot \nabla_x \phi_c) (n\cdot v) f  \dd v \dd S_x \notag \\
  & = 2 \int_{\p\O_2}\int_{n(x)\cdot v>0} \big(|v|^2 -5 \big) \sqrt{\mu} |n(x)\cdot v|^2 (n(x)\cdot \nabla \phi_c) f \dd v \dd S_x \notag\\
  & = -2\int_{\p\O_2}\int_{n(x)\cdot v>0} \big(|v|^2 -5 \big) \sqrt{\mu} |n(x)\cdot v|^2 \frac{\beta(x)}{2-\beta(x)}\phi_c(x) f \dd v \dd S_x = 0. \label{J2_c_first}
\end{align}
In the second line we have applied the change of variable $v\to v-2(n\cdot v)n$ in \eqref{cov_spec}, and the contribution of the first part $v$ in $ v-2(n\cdot v)n$ vanishes by applying the specular boundary condition on $x\in \p\O_2$. In the third line we have applied the boundary condition of $\phi_c$ in~\eqref{poisson_c}. In the last equality, we used $\beta(x)\leq \alpha(x)=0$ for $x\in \p\O_2$ from~\eqref{portion}.

Next we compute the diffuse reflection part, i.e, $x\in \p\O_1$. We have
\begin{align}
 & \int_{\p\O_1} \bigg[ \int_{n(x)\cdot v>0} + \int_{n(x)\cdot v<0}\bigg] \big(|v|^2 -5\big)\sqrt{\mu} (v\cdot \nabla_x \phi_c) (n\cdot v) f  \dd v \dd S_x \notag \\
 & = \int_{\p\O_1} \int_{n(x)\cdot v>0} \big(|v|^2 -5 \big) \sqrt{\mu} (v\cdot \nabla_x \phi_c) (n\cdot v)(f-P_\gamma f) \dd v \dd S_x \notag\\
  & + 2 \int_{\p\O_1}\int_{n(x)\cdot v>0} \big(|v|^2 -5 \big) \sqrt{\mu} |n(x)\cdot v|^2 (n(x)\cdot \nabla \phi_c) P_\gamma f \dd v \dd S_x \notag    \\
  & = \int_{\p\O_1} \int_{n(x)\cdot v>0} \big(|v|^2 -5 \big) \sqrt{\mu} (v\cdot \nabla_x \phi_c) (n\cdot v)(f-P_\gamma f) \dd v \dd S_x \notag \\
  & \lesssim \delta_1 \Vert \nabla \phi_c\Vert_{L^2(\p\O)}^2 +  \frac{1}{\delta_1} |\mathbf{1}_{x\in \p\O_1}(1-P_\gamma)f|_{2,+}^2  \lesssim \delta_1 \Vert c\Vert_{L^2_x}^2 + \frac{1}{\delta_1}|\mathbf{1}_{x\in \p\O_1}(1-P_\gamma)f|_{2,+}^2   .\label{J2_c_second}
\end{align}
In the first equality we have applied the change of variable $v\to v-2(n\cdot v)n$ in \eqref{cov_diffuse}. The third line vanishes by $\int_{n(x)\cdot v>0} (|v|^2-5)(n\cdot v)^2\mu \dd v= 0$. In the last inequality we applied~\eqref{poisson_h2} to~\eqref{poisson_c} with the trace theorem:
\begin{align*}
    & \Vert \nabla \phi_c \Vert^2_{L^2(\p\O)} \lesssim \Vert \phi_c \Vert^2_{H^2_x} \lesssim \Vert c\Vert_{L^2_x}^2.
\end{align*}

Collecting~\eqref{J2_c_first} and~\eqref{J2_c_second}, we conclude the estimate for $J_2$ as
\begin{equation}\label{J2_c_bdd}
|J_2| \lesssim \delta_1 \int_s^t \Vert c\Vert_{L^2_x}^2 \dd \tau + \frac{1}{\delta_1}\int_s^t |\mathbf{1}_{x\in \p\O_1}(1-P_\gamma)f|_{2,+}^2  \dd \tau.
\end{equation}

For $J_3$, due to the exponential decay factor $\mu^{1/2}$ in $\phi_c$, we have
\begin{align}
  |J_3|  &  \lesssim  \delta_1 \int_s^t \Vert c\Vert_{L^2_x}^2 \dd \tau + \frac{1}{\delta_1}\int_s^t \Vert \mu^{1/4} \mathcal{L} (\mathbf{I}-\mathbf{P})f\Vert_{L^2}^2 \dd \tau \notag\\
  &\lesssim  \delta_1 \int_s^t \Vert c\Vert_{L^2_x}^2 \dd \tau + \frac{1}{\delta_1} \int_s^t \Vert (\mathbf{I}-\mathbf{P})f\Vert_{L^2}^2 \dd \tau . \label{J3_c_bdd}
\end{align}
Here again we applied \eqref{poisson_h2} for $\nabla \phi_c$.

For the contribution of $g$ in $J_4$, similar to the computation in~\eqref{J3_c_bdd}, we have
\begin{align}
  \Big|\int_s^t \int_{\O\times \mathbb{R}^3} g\psi\dd x \dd v \dd \tau\Big|  & \lesssim \delta_1 \int_s^t \Vert c\Vert_{L^2_x}^2 \dd \tau + \frac{1}{\delta_1} \int_s^t \Vert g\Vert_{L^2}^2 \dd \tau. \label{J4_c_bdd}
\end{align}

For the contribution of $\lambda f$ in $J_4$, with the choice of $\phi$ in \eqref{test_c}, we apply \eqref{poisson_h2} to have
\begin{align}
    & \Big|\int_s^t \int_{\O\times \mathbb{R}^3} \lambda f\psi\dd x \dd v \dd \tau\Big| \lesssim \delta_1 \int_s^t \Vert \phi_c \Vert_{H^1_x}^2 \dd \tau + \frac{1}{\delta_1} \int_s^t \Vert \mu^{1/4} (\mathbf{I}-\mathbf{P})f \Vert_{L^2}^2 \dd \tau \notag\\
    & \lesssim \delta_1 \int_s^t \Vert c\Vert_{L^2_x}^2 \dd \tau+ \frac{1}{\delta_1} \int_s^t \Vert \mu^{1/4} (\mathbf{I}-\mathbf{P})f \Vert_{L^2}^2 \dd \tau . \label{J4_c_bdd_2}
\end{align}

Collecting~\eqref{LHS_c}, \eqref{J1_c_bdd}, \eqref{J2_c_bdd}, \eqref{J3_c_bdd}, \eqref{J4_c_bdd} and \eqref{J4_c_bdd_2}, we conclude the estimate of $c$ as follows: for some $C_1>0$,
\begin{align}
  \int_s^t \Vert c\Vert_{L^2_x}^2 \dd \tau & \leq C_1\Big[ G_c(t) - G_c(s) + \delta_1 \int_s^t \Vert \mathbf{b}\Vert_{L^2_x}^2 \dd \tau + \frac{1}{\delta_1} \int_s^t \Vert (\mathbf{I}-\mathbf{P})f\Vert_{L^2}^2 \dd \tau \notag\\
  & + \frac{1}{\delta_1}\int_s^t |\mathbf{1}_{x\in \p\O_1}(1-P_\gamma)f|_{2,+}^2 \dd \tau +  \frac{1}{\delta_1} \int_s^t \Vert g\Vert_{L^2}^2 \dd \tau\Big]. \label{c_bdd}
\end{align}

\textbf{Step 2: estimate of $\mathbf{b}(x)$.}
First we choose test function as
\Be\notag
 \psi_j = \phi_j({\color{black}t},x)\chi_j \ \  {\color{black}\text{for} \ \  j=1,2,3},
\Ee
and will take the test function in~\eqref{weak_formula} as
\begin{equation*}
\psi = \sum_{j=1}^3 \psi_j = \sum_{j=1}^3 \phi_j(t,x)\chi_j.
\end{equation*}

{\color{black}Taking $s=t-\triangle$ and $ t=t $, the} LHS of~\eqref{weak_formula} becomes
\begin{align*}
   \text{LHS of~\eqref{weak_formula}}& = \sum_{j=1}^3 \int_{\O} \{ {\color{black} b_j(t ) - b_j(t-\triangle)}\} \phi_j \dd x. 
\end{align*}

From property of $\mathcal{L}$, we have $I_3=0$, and the contribution of $g$ in $I_4$ also vanishes. For the contribution of $\lambda f$ in $I_4$, we have
\begin{align}
    & \int_{\O\times \mathbb{R}^3} \lambda f \psi \dd x \dd v = \lambda\sum_{i=1}^3 \int_{\O} \phi_j(x)b_j(x) \dd x \leq \lambda \sum_{j=1}^3 [\Vert \phi_j\Vert_{L^2_x}^2 + \Vert b_j\Vert_{L^2_x}^2 ].   \label{I4_b_bdd}
\end{align}

Then {\color{black}for each $j=1,2,3,$}
\Be\begin{split}\label{v_nabla_psi}
    &  v\cdot \nabla_x \psi_j =  \sum_{i=1}^3 \p_i \phi_j(x) v_i v_j \mu^{1/2} \\
    & =  \Big[ \sum_{i=1}^3 \p_i \phi_j(x) \mu^{1/2}\left(v_iv_j - \frac{\delta_{ij}}{3}|v|^2\right) + \sum_{i=1}^3 \p_i \phi_j(x) \mu^{1/2} \frac{\delta_{ij}}{3}|v|^2  \Big] \\
    & =   \sum_{i=1}^3 \p_i\phi_j   \hat{A}_{ij}     + \p_j \phi_j \mu^{1/2} \left(\frac{|v|^2}{3}-1\right) + \p_j \phi_j \mu^{1/2}    \\
    & =    \p_j \phi_j  \left(\chi_0 + \frac{\sqrt{6}}{3}\chi_4\right)+ \sum_{i=1}^3 \p_i \phi_j \hat{A}_{ij} ,
\end{split}
\Ee
{\color{black} where we have used $\hat A_{ij}$ in \eqref{A_j}.}

For RHS of~\eqref{weak_formula}, from~\eqref{A_B_proj_0} and \eqref{v_nabla_psi}, we have
\begin{align*}
   I_1 &  =  \sum_{j=1}^3 \int_{t-\triangle}^{t} \int_{\O}  \{\p_j \phi_j (a+\frac{\sqrt{6}}{3} c ) \} \dd x \dd \tau \\
   & +  \sum_{j=1}^3 \int_{t-\triangle}^{t} \int_{\O} \sum_{i=1}^3  \p_i \phi_j  \langle \hat{A}_{ij} , (\mathbf{I}-\P)f \rangle \dd x \dd \tau.
\end{align*}
Taking ${\color{black}\triangle} \to 0 $ yields
\begin{align}
& \sum_{j=1}^3 \int_{\O} \phi_j(t,x) \p_\tau b_j(t,x)   \dd x    =   \sum_{j=1}^3\int_{\O}  \{\p_j \phi_j (t,x)  (a (t,x)  +\frac{\sqrt{6}}{3} c (t,x)   ) \} \dd x \notag\\
  & + \sum_{j=1}^3 \int_{\O} \sum_{i=1}^3  \p_i \phi_j  \langle \hat{A}_{ij} ,(\mathbf{I}-\P)f (t,x)  \rangle \dd x - \sum_{j=1}^3\int_\gamma  f (t,x)  \phi_j (t,x)  \chi_j   \dd \gamma. \label{delta_to_0}
\end{align}

Let $\mathbf{\Phi_b} = (\Phi_b^1, \Phi_b^2, \Phi_b^3)$ solve
\begin{align}
  -\Div(\nablaS \mathbf{\Phi_b})  &    = \frac{1}{2}\p_\tau  \mathbf{b}(t) \text{ in }\O \notag
  \\
  \mathbf{\Phi_b} \cdot n& = 0 \label{sym_poisson_b_partial} \text{ on }\p\O\\
 (2-\beta(x)) \big[ \nablaS \mathbf{\Phi_b}  n  - (\nablaS \mathbf{\Phi_b} : n \otimes n)n \big] + \beta(x)\mathbf{\Phi_b} &= 0 \text{ on }\p\O. \notag 
\end{align}

By Lemma \ref{lemma:sym_poisson}, there exists a unique solution $\mathbf{\Phi_b} \in \mathcal{X}$ defined in~\eqref{hilbert_X} such that $\Vert \mathbf{\Phi_b} \Vert_{H^2_x} \lesssim \Vert \p_\tau \mathbf{b}\Vert_{L^2_x}$ and $\mathbf{\Phi_b}$ satisfies~\eqref{sym_poisson_b_partial} a.e. We choose $\phi_j = \Phi_b^j$. Then we compute the contribution of the boundary integral. First we compute the specular part, i.e, $x\in \p\O_2$, we have
\begin{align*}
&  \int_{\p\O_2} \bigg[ \int_{n(x)\cdot v>0} + \int_{n(x)\cdot v<0}\bigg] f(x,v) (\mathbf{\Phi_b} \cdot v) \sqrt{\mu}(n(x)\cdot v) \dd v \dd S_x  \notag\\
& = 2\int_{\p\O_2} \int_{n(x)\cdot v>0} f |n(x)\cdot v|^2 \sqrt{\mu} \mathbf{\Phi_b} \cdot n \dd v \dd S_x = 0. 
\end{align*}
In the second line we applied the change of variable $v\to v-2(n(x)\cdot v)n(x)$ in \eqref{cov_spec}, and the contribution of the first part $v$ vanishes due to the specular boundary condition. In the last equality we applied the first boundary condition in \eqref{sym_poisson_b_partial}: $\mathbf{\Phi_b} \cdot n = 0$ on $\p\O$.

Then we compute the diffuse reflection part, i.e, $x\in \p\O_1$. We have
\begin{align*}
&  \int_{\p\O_1} \bigg[ \int_{n(x)\cdot v>0} + \int_{n(x)\cdot v<0}\bigg] f(x,v) (\mathbf{\Phi_b} \cdot v) \sqrt{\mu}(n(x)\cdot v) \dd v \dd S_x  \notag\\
& = \int_{\p\O_1} \int_{n(x)\cdot v > 0}  (\mathbf{\Phi_b} \cdot v) \sqrt{\mu} (n(x)\cdot v) (1-P_\gamma) f \dd v \dd S_x \notag\\
& \lesssim  o(1)\Vert \mathbf{\Phi_b} \Vert_{L^2(\p\O)}^2 + |\mathbf{1}_{x\in \p\O_1}(1-P_\gamma)f|_{2,+}^2.       
\end{align*}
In the second line, we applied the change of variable $v\to v-2(n(x)\cdot v)n(x)$ in \eqref{cov_diffuse}, and the contribution of the second part $-2(n(x)\cdot v)n(x)$ vanishes due to the boundary condition $\mathbf{\Phi_b}\cdot n=0$.

We conclude the estimate for $I_2$ as
\begin{align}
  |I_2|  &  \lesssim  o(1)\Vert \mathbf{\Phi_b} \Vert_{L^2(\p\O)}^2 + |\mathbf{1}_{x\in \p\O_1}(1-P_\gamma)f|_{2,+}^2.  \label{I2_b_bdd}
\end{align}

The symmetric Poisson system~\eqref{sym_poisson_b_partial} leads to
\begin{align*}
  & \frac{1}{2}\sum_{j=1}^3 \int_{\O}\Phi_b^j \p_\tau b_j \dd x  = - \int_{\O} \Div(\nablaS \mathbf{\Phi_b}) \cdot   \mathbf{\Phi_b} \dd x   \\
  & =  - \int_{\p\O} (\nablaS \mathbf{\Phi_b} n) \mathbf{\Phi_b} \dd S_x + \int_{\O} |\nablaS \mathbf{\Phi_b}|^2 \dd x \\
  & = -\int_{\p\O} (\nablaS \mathbf{\Phi_b} : n(x)\otimes n(x)) (\mathbf{\Phi_b}\cdot n)\dd S_x + \int_{\p\O} \frac{\beta(x)}{2-\beta(x)} |\mathbf{\Phi_b}|^2 \dd S_x+  \int_{\O} |\nablaS \mathbf{\Phi_b}|^2 \dd x\\
   & =\int_{\p\O} \frac{\beta(x)}{2-\beta(x)} |\mathbf{\Phi_b}|^2 \dd S_x+  \int_{\O} |\nablaS \mathbf{\Phi_b}|^2 \dd x.
\end{align*}
In the third line we applied the second boundary condition in~\eqref{sym_poisson_b_partial}. In the last line we applied the first boundary condition~\eqref{sym_poisson_b_partial}.

Combining~\eqref{delta_to_0}, \eqref{I4_b_bdd} and \eqref{I2_b_bdd}, with applying trace theorem to \eqref{I2_b_bdd}, we have
\begin{align*}
    &\Vert \mathbf{\Phi_b} \Vert_{H^1_x}^2 \lesssim \int_{\p\O} \frac{\beta(x)}{2-\beta(x)} |\mathbf{\Phi_b}|^2 \dd S_x+  \int_{\O} |\nablaS \mathbf{\Phi_b}|^2 \dd x = \frac{1}{2}\sum_{j=1}^3 \int_{\O}\Phi_b^j \p_\tau b_j \dd x \notag\\
    &  \lesssim \Vert \nabla \mathbf{\Phi_b}\Vert_{L^2_x}(\Vert a\Vert_{L^2_x} + \Vert c\Vert_{L^2_x} + \Vert \mu^{1/4}(\mathbf{I}-\mathbf{P})f\Vert_{L^2}) + o(1)\Vert \mathbf{\Phi_b}\Vert^2_{H^1_x} + o(1)\Vert b\Vert_{L^2_x}^2 + |\mathbf{1}_{x\in \p\O_1}(1-P_\gamma)f|^2_{2,+}.
\end{align*}
In the first inequality we applied the Korn's inequality~\eqref{korn} since $\beta(x)$ satisfies the condition in Lemma \ref{lemma:sym_poisson}.

Hence we obtain the estimate for $\mathbf{\Phi_b}$ as
\begin{equation}\label{test_b_partial_H1}
\Vert \mathbf{\Phi_b}\Vert_{H^1_x} \lesssim \Vert a\Vert_{L^2_x} + o(1)\Vert \mathbf{b}\Vert_{L^2_x} + \Vert c\Vert_{L^2_x} + \Vert \mu^{1/4}(\mathbf{I}-\mathbf{P})f\Vert_{L^2} +   |\mathbf{1}_{x\in \p\O_1}(1-P_\gamma)f|_{2,+}.
\end{equation}

Next we use the weak formulation in~\eqref{weak_formula_2} for the estimate of $\mathbf{b}$. 

Let $\mathbf{\phi_b}=(\phi_b^1,\phi_b^2,\phi_b^3)$ be the solution of the following system
\begin{align}
  -\Div(\nablaS \mathbf{\phi_b})  & = \frac{1}{2}\mathbf{b}(\tau) \text{ in }\O \notag \\
  \mathbf{\phi_b} \cdot n & = 0  \text{ on } \p\O\label{sym_poisson_b}\\
 (2-\beta(x))\big[ \nablaS \mathbf{\phi_b} n - (\nablaS \mathbf{\phi_b} : n\otimes n)n\big] + \beta(x)\mathbf{\phi_b} & = 0 \text{ on } \p\O .\notag
\end{align}
By Lemma \ref{lemma:sym_poisson}, there is a unique solution $\mathbf{\phi_b}$ satisfying the system~\eqref{sym_poisson_b} with
\begin{equation}\label{phi_b_H2}
\Vert \mathbf{\phi_b}\Vert_{H^2_x}\lesssim \Vert \mathbf{b}\Vert_{L^2_x}.
\end{equation}

We choose test function as
\begin{align}
  &\psi=\psi_b \notag\\
  &:= \sum_{i,j=1}^3  \p_j \phi^i_b v_i v_j \mu^{1/2} - \sum_{i=1}^3  \p_i \phi_b^i  \mu^{1/2} \label{test_b_1}\\
  & = \sum_{i,j=1}^3 \p_j \phi^i_b \mu^{1/2} [v_i v_j - \frac{\delta_{ij}}{3}|v|^2] + \sum_{i,j=1}^3 \p_j \phi^i_b \mu^{1/2} \frac{\delta_{ij}}{3}|v|^2 - \sum_{i=1}^3 \p_i \phi_b^i \mu^{1/2} \notag\\
  & = \sum_{i,j=1}^3 \p_j \phi_b^i  \hat{A}_{ij} + \sum_{i=1}^3 \p_i \phi^i_b \mu^{1/2} \Big[\frac{|v|^2-3}{3} \Big] \notag\\
  & = \sum_{i,j=1}^3 \p_j \phi_b^i  \hat{A}_{ij} + \sum_{i=1}^3 \p_i \phi_b^i  \chi_4 \frac{\sqrt{6}}{3}. \label{test_b_2}
\end{align}
Here $\chi_4$ and $\hat{A}_{ij}$ are defined in~\eqref{velocity_inner} and~\eqref{A_j} respectively.

We compute the transport operator $-v\cdot \nabla_x$ on $\psi_b$ using~\eqref{test_b_1}:
\begin{align}
    & -v\cdot \nabla_x \psi_b = -\sum_{i,j,k=1}^3 \p_{kj}\phi_b^i v_i v_j v_k \mu^{1/2} + \sum_{i,k=1}^3 v_k\p_{ki} \phi_b^i \mu^{1/2}  \notag \\
    & = -\sum_{i,j,k=1}^3 \p_{kj} \phi_b^i (\mathbf{I}-\mathbf{P})(v_i v_j v_k \mu^{1/2}) - \sum_{i,j,k=1}^3  \p_{kj} \phi_b^i  \mathbf{P}(v_i v_j v_k \mu^{1/2}) + \sum_{i,k=1}^3 v_k\p_{ki}\phi_b^i \mu^{1/2}. \label{test_b_derivative}
\end{align}
For $\mathbf{P}(v_iv_j v_k \mu^{1/2})$, when $i=j=k$, we have
\begin{align*}
   \mathbf{P}((v_i)^3 \mu^{1/2}) & = \langle (v_i)^3 \mu^{1/2}, \chi_i \rangle \chi_i = 3\chi_i,
\end{align*}
where we applied the computation in~\eqref{fact}.

When $i=j\neq k$, we apply~\eqref{fact} to have
\begin{align*}
   \mathbf{P}((v_i)^2 v_k \mu^{1/2}) & = \langle (v_i)^2 v_k \mu^{1/2}, \chi_k \rangle \chi_k = \chi_k.
\end{align*}
When $j=k\neq i$, we apply~\eqref{fact} to have
\begin{align*}
   \mathbf{P}((v_j)^2 v_i \mu^{1/2}) & = \langle (v_j)^2 v_i \mu^{1/2}, \chi_i \rangle \chi_i = \chi_i.
\end{align*}
When $i=k\neq j$, we apply~\eqref{fact} to have
\begin{align*}
   \mathbf{P}((v_i)^2 v_j \mu^{1/2}) & = \langle (v_i)^2 v_j \mu^{1/2}, \chi_j \rangle \chi_j = \chi_j.
\end{align*}
The above computation yields
\begin{align}
   &-\sum_{i,j,k=1}^3 \p_{kj}\phi^i_b \mathbf{P}(v_iv_jv_k \mu^{1/2}) \notag\\ 
   &=-3\sum_{ {\color{black}i=j=k}=1}^3 \p_{ii}\phi_b^i \chi_i-\sum_{j=k\neq i} \p_{jj}\phi^i_b \chi_i -\sum_{i=j\neq k} \p_{ki}\phi^i_b \chi_k -\sum_{i=k\neq j} \p_{ij}\phi^i_b \chi_j \notag\\
   &=-3\sum_{i=1}^3 \p_{ii}\phi_b^i \chi_i-\sum_{j\neq i} \p_{jj}\phi^i_b \chi_i -\sum_{i\neq k} \p_{ki}\phi^i_b \chi_k -\sum_{i\neq j} \p_{ij}\phi^i_b \chi_j .  \label{estimate_P}
\end{align}
Here we note that the last two terms are the same.

The last term in~\eqref{test_b_derivative} reads
\begin{align}
    &  \sum_{i=k}^3 \p_{ii} \phi_b^i  \chi_i + \sum_{i\neq k}^3 \p_{ki}\phi_b^i  \chi_k. \label{estimate_second}
\end{align}

Collecting~\eqref{estimate_P} and~\eqref{estimate_second}, the last two terms in~\eqref{test_b_derivative} combine to be
\begin{align*}
    &- \sum_{i,j,k=1}^3  \p_{kj} \phi_b^i  \mathbf{P}(v_i v_j v_k \mu^{1/2}) + \sum_{i,k=1}^3 v_k\p_{ki}\phi_b^i \mu^{1/2} \notag\\
    & = -2\sum_{i=1}^3 \p_{ii} \phi^i_b 
    \chi_i
    - \sum_{j\neq i} \p_{jj} \phi^i_b  \chi_i - \sum_{i\neq j}\p_{ij}\phi_b^i \chi_j \notag \\
    & = -\sum_{i=1}^3 \p_{ii} \phi^i_b \chi_i - \sum_{j\neq i} \p_{jj} \phi^i_b  \chi_i -\sum_{i=1}^3 \p_{ii} \phi^i_b  \chi_i- \sum_{{\color{black}j}\neq {\color{black}i}}\p_{{\color{black}ji}}\phi_b^{{\color{black}j}}\chi_{{\color{black}i}} \notag\\
    & =-\sum_{i=1}^3 \chi_i  \sum_{j=1}^3 \p_{jj}\phi_b^i    - \sum_{i=1}^3 \chi_i  \sum_{j=1}^3 \p_{ij}\phi_b^j \notag\\
    & = -\sum_{i=1}^3  \chi_i [\Delta \phi^i_b + \p_i \Div(\mathbf{\phi_b})] = \sum_{i=1}^3 \chi_i b_i.  
\end{align*}
In the last line we used that $\phi_b^i$ is the solution to the system~\eqref{sym_poisson_b}, and $\Div(\nablaS \mathbf{\phi_b}) = \frac{1}{2} (\Delta \mathbf{\phi_b} + \nabla \Div \mathbf{\phi_b})$ from~\eqref{div_Delta}.

Then for~\eqref{test_b_derivative} we conclude that
\begin{align*}
   \eqref{test_b_derivative} & = \sum_{i=1}^3  \chi_i b_i -\sum_{i,j,k=1}^3 \p_{kj} \phi_b^i  (\mathbf{I}-\mathbf{P})(v_i v_j v_k \mu^{1/2}).
\end{align*}

Thus LHS of~\eqref{weak_formula_2} becomes
\begin{align}
   LHS & =    \int_s^t  \int_{\O} |\mathbf{b}|^2 \dd x \dd \tau -   \underbrace{\sum_{i,j,k=1}^3 \int_s^t  \int_\O \p_{kj}\phi_b^i \langle v_iv_jv_k \mu^{1/2}, (\mathbf{I}-\mathbf{P})f \rangle  \dd x \dd \tau}_{E_2}  , \label{LHS_b}
   \end{align}
where, by~\eqref{phi_b_H2}, for some $\delta_2\ll 1$, 
\begin{align*}
  |E_2|  & \lesssim    \delta_2 \int_s^t \Vert \mathbf{b}\Vert_{L^2_x}^2\dd \tau  + \frac{1}{\delta_2}  \int_s^t \Vert \mu^{1/4}(\mathbf{I}-\mathbf{P})f\Vert^2_{L^2} \dd \tau. 
\end{align*}

Next we estimate $J_i, 1\leq i\leq 4$ in~\eqref{weak_formula_2}. Note that $\mathbf{\Phi_b} = \p_\tau \mathbf{\phi_b}$, where we have an estimate of $\mathbf{\Phi_b}$ in~\eqref{test_b_partial_H1}. Applying the second representation of $\psi_b$ in~\eqref{test_b_2}, from~\eqref{A_j}, we have 
\begin{align}
 |J_1|   &  \lesssim  \int_s^t  \int_{\O} \big( |c(x)| + |\langle \hat{A}_{ij}, (\mathbf{I}-\mathbf{P})f\rangle|\big) |\nabla_x \mathbf{\Phi_b}| \dd x \dd \tau \notag\\
    &\lesssim \delta_2 \int_s^t \Vert \nabla \mathbf{\Phi_b}\Vert^2_{L^2_x} \dd \tau + \frac{1}{\delta_2} \int_s^t [\Vert c\Vert_{L^2_x}^2+\Vert \mu^{1/4}(\mathbf{I}-\mathbf{P})f\Vert^2_{L^2}] \dd \tau \notag\\
    &\lesssim \delta_2 \int_s^t \Vert \mathbf{b}\Vert_{L^2_x}^2 \dd \tau + \delta_2 \int_s^t \Vert a\Vert_{L^2_x}^2\dd \tau + \frac{1}{\delta_2} \int_s^t \Vert c\Vert_{L^2_x}^2\dd \tau + \frac{1}{\delta_2} \int_s^t \Vert \mu^{1/4}(\mathbf{I}-\mathbf{P})f \Vert^2_{L^2} \dd \tau. \label{J1_b_bdd}
\end{align}
In the last line we used~\eqref{test_b_partial_H1}.

Then we focus on the boundary integral $J_2$. We use the representation of $\psi_b$ in~\eqref{test_b_1} to have
\begin{align}
 \int_\gamma \psi_b f \dd \gamma   = &   \int_\gamma (n\cdot v)\Big(\sum_{i,j=1}^3  \p_j \phi^i_b v_i v_j \mu^{1/2} - \sum_{i=1}^3  \p_i \phi_b^i  \mu^{1/2}\Big)f\dd v \dd S_x. \label{j2_b_two_terms}
\end{align}
We begin by computing the second term. First we compute the specular part, i.e, $x\in \p\O_2$. We have
\begin{align}
    & - \int_{\p\O_2} \bigg[ \int_{n(x)\cdot v>0}    +     \int_{n(x)\cdot v<0}  \bigg] \Div (\mathbf{\phi_b}) \mu^{1/2} f(x,v) (n\cdot v) \dd v \dd S_x  = 0,  \label{j2_b_second_specular}
\end{align}
where we have applied the change of variable $v\to v-(2n(x)\cdot v)n(x)$ in \eqref{cov_spec}.

Then we compute the diffuse reflection part, i.e, $x\in \p\O_1$. By the change of variable in \eqref{cov_diffuse} we have
\begin{align}
    & \Big|- \int_{\p\O_1} \bigg[ \int_{n(x)\cdot v>0}    +     \int_{n(x)\cdot v<0}  \bigg] \Div (\mathbf{\phi_b}) \mu^{1/2} f(x,v) (n\cdot v) \dd v \dd S_x  \Big| \notag \\
    & =\Big|-\int_{\p\O_1} \int_{n(x)\cdot v>0} \Div(\mathbf{\phi_b}) \mu^{1/2} f(x,v) (n\cdot v) (1-P_\gamma) f \dd v \dd S_x \Big|\notag\\
    &\lesssim o(1)\Vert \nabla \mathbf{\phi_b} \Vert_{L^2_x}^2   +  |\mathbf{1}_{x\in \p\O_1}(1-P_\gamma)f|^2_{2,+} \lesssim o(1)\Vert \mathbf{b}\Vert_{L^2_x}^2 + |\mathbf{1}_{x\in \p\O_1}(1-P_\gamma)f|^2_{2,+}. \label{j2_b_second_diffuse}
\end{align}
In the last line we used~\eqref{phi_b_H2}.

Next we compute the first term in~\eqref{j2_b_two_terms}. We first compute the specular part, i.e, $x\in \p\O_2$. Through the change of variable $v\to v-2(n\cdot v)n$ in \eqref{cov_spec}, we have
\begin{align}
 & \int_{\p\O_2} \bigg[\int_{n(x)\cdot v>0} + \int_{n(x)\cdot v<0} \bigg]  (n\cdot v) \mu^{1/2} f(x,v)  \big[ \nablaS \mathbf{\phi_b} : (v\otimes v) \big]  \dd v \dd S_x   \notag\\
 &  =  4\int_{\p\O_2} \int_{n(x)\cdot v>0}   |n\cdot v|^2 \mu^{1/2} f(x,v) \big[\nablaS \mathbf{\phi_b} : (n\otimes v)  - \nablaS \mathbf{\phi_b} : (n\otimes n) (n\cdot v)\big]  \notag\\
 & = -4 \int_{\p\O_2} \int_{n(x)\cdot v>0} |n\cdot v|^2 \mu^{1/2} f(x,v) \frac{\beta(x)}{2-\beta(x)} (\mathbf{\phi_b} \cdot v) = 0. \label{j2_b_first_specular}
\end{align}    
In the second line, the contribution of $v\otimes v$ vanishes from the specular boundary condition of $f$, and we used that $\nablaS \mathbf{\phi_b}$ is symmetric so that $\nablaS \mathbf{\phi_b} : (n\otimes v) = \nablaS \mathbf{\phi_b} : (v\otimes n)$. In the last line, we have applied the second boundary condition in \eqref{sym_poisson_b}, and we used the fact that $\beta(x)=0$ when $x\in \p\O_2$.

Then we compute the diffuse part, i.e, $x\in \p\O_1$. We apply the change of variable \eqref{cov_diffuse} to have
\begin{align}
 & \int_{\p\O_1} \bigg[\int_{n(x)\cdot v>0} + \int_{n(x)\cdot v<0} \bigg]  (n\cdot v) \mu^{1/2} f(x,v)  \big[ \nablaS \mathbf{\phi_b} : (v\otimes v) \big]  \dd v \dd S_x   \notag\\
 & =   \int_{\p\O_1}  \int_{n(x)\cdot v>0} (n\cdot v) \mu^{1/2} \big[ \nablaS \mathbf{\phi_b} : (v\otimes v) \big]  (1-P_\gamma)f \notag\\
 &  +  4\int_{\p\O_1} \int_{n(x)\cdot v>0}   |n\cdot v|^2 \mu^{1/2} P_\gamma f \big[\nablaS \mathbf{\phi_b} : (n\otimes v)  - \nablaS \mathbf{\phi_b} : (n\otimes n) (n\cdot v)\big]  \notag\\
 & =  \int_{\p\O_1}  \int_{n(x)\cdot v>0} (n\cdot v) \mu^{1/2} \big[ \nablaS \mathbf{\phi_b} : (v\otimes v) \big]  (1-P_\gamma)f \notag \\
 &  -4 \int_{\p\O_1} \int_{n(x)\cdot v>0} |n\cdot v|^2 \mu^{1/2} P_\gamma f \frac{\beta(x)}{2-\beta(x)} \mathbf{\phi_b} \cdot v \notag \\
 & =  \int_{\p\O_1}  \int_{n(x)\cdot v>0} (n\cdot v) \mu^{1/2} \big[ \nablaS \mathbf{\phi_b} : (v\otimes v) \big]  (1-P_\gamma)f \notag \\
 &  -4 \int_{\p\O_1} \int_{n(x)\cdot v>0} |n\cdot v|^2 \mu^{1/2} P_\gamma f \frac{\beta(x)}{2-\beta(x)} \mathbf{\phi_b} \cdot (v - (n\cdot v)n + (n\cdot v)n) \notag \\
 & =  \int_{\p\O_1}  \int_{n(x)\cdot v>0} (n\cdot v) \mu^{1/2} \big[ \nablaS \mathbf{\phi_b} : (v\otimes v) \big]  (1-P_\gamma)f.\label{j2_b_first_diffuse}
\end{align}    
In the last line, the term $(n\cdot v)n$ vanishes due to the boundary condition of $\mathbf{\phi_b}\cdot n = 0$ in~\eqref{sym_poisson_b}. The remaining term $v-(n\cdot v)n$ vanishes since $|n\cdot v|^2 \mu \mathbf{\phi_b} \cdot (v-(n\cdot v)n)$ is odd in tangential direction of $n\cdot v$.

Collecting~\eqref{j2_b_first_diffuse}, \eqref{j2_b_first_specular}, \eqref{j2_b_second_diffuse} and \eqref{j2_b_second_specular} we obtain the estimate for $J_2$ as
\begin{align}
    |J_2| \lesssim o(1)\Vert \mathbf{b}\Vert_{L^2_x}^2 + |\mathbf{1}_{x\in \p\O_1}(1-P_\gamma)f|^2_{2,+}. \label{J2_b_bdd}
\end{align}
Here we used $\Vert \nablaS \mathbf{\phi_b}\Vert_{L^2(\p\O)} \lesssim \Vert  \mathbf{\phi_b} \Vert_{H^1(\p\O)} \lesssim \Vert \mathbf{\phi_b}\Vert_{H^2_x} \lesssim \Vert \mathbf{b}\Vert_{L^2_x}$ from \eqref{phi_b_H2} and trace theorem.

For $J_3$ and the contribution of $g$ in $J_4$, similar to~\eqref{J3_c_bdd} and~\eqref{J4_c_bdd}, we have
\begin{align}
  |J_3|  &  \lesssim \delta_2 \int_s^t \Vert \mathbf{b}\Vert_{L^2_x}^2\dd \tau + \frac{1}{\delta_2} \int_s^t \Vert (\mathbf{I}-\mathbf{P})f \Vert^2_{L^2} \dd \tau, \label{J3_b_bdd}
\end{align}
\begin{align}
  \Big|\int_s^t \int_{\O\times \mathbb{R}^3} g\psi \dd x \dd v \dd \tau \Big|  &  \lesssim \delta_2 \int_s^t \Vert \mathbf{b}\Vert_{L^2_x}^2\dd \tau +\frac{1}{\delta_2} \int_s^t \Vert g\Vert^2_{L^2} \dd \tau. \label{J4_b_bdd}
\end{align}

For the contribution of $\lambda f$ in $J_4$, with the choice of $\psi$ in \eqref{test_b_2}, we apply \eqref{sym_poisson_h2} to have
\begin{align}
  \Big|\lambda \int_s^t \int_{\O\times \mathbb{R}^3} f\psi \dd x \dd v \dd \tau \Big|    & \lesssim   \delta_2 \int_s^t \Vert \mathbf{\phi_b}\Vert_{H^1_x}^2 \dd \tau+ \frac{1}{\delta_2} \int_s^t [\Vert c\Vert_{L^2_x}^2 + \Vert \mu^{1/4} (\mathbf{I}-\mathbf{P})f\Vert_{L^2}^2] \dd \tau     \notag \\
  & \lesssim \delta_2 \int_s^t \Vert \mathbf{b}\Vert_{L^2_x}^2 \dd \tau + \frac{1}{\delta_2} \int_s^t [\Vert c\Vert_{L^2_x}^2 + \Vert \mu^{1/4} (\mathbf{I}-\mathbf{P})f\Vert_{L^2}^2] \dd \tau \label{J_4_b_bdd_2}. 
\end{align}

Collecting \eqref{LHS_b}, \eqref{J1_b_bdd}, \eqref{J2_b_bdd}, \eqref{J3_b_bdd},  \eqref{J4_b_bdd} and \eqref{J_4_b_bdd_2}, we conclude the estimate for $\mathbf{b}$ as follows: for some $C_2>0$,
\begin{align}
  \int_s^t \Vert \mathbf{b}\Vert_{L^2_x}^2 \dd \tau  & \leq C_2 \Big[ G_b(t) - G_b(s) + \delta_2 \int_s^t \Vert a\Vert_{L^2_x}^2 \dd \tau + \frac{1}{\delta_2} \int_s^t \Vert c\Vert_{L^2_x}^2 \dd \tau \notag\\
  & +  \frac{1}{\delta_2} \int_s^t \Vert (\mathbf{I}-\mathbf{P})f\Vert_{L^2}^2 \dd \tau + \frac{1}{\delta_2}\int_s^t \Vert g\Vert_{L^2}^2 \dd \tau + \frac{1}{\delta_2}\int_s^t|\mathbf{1}_{x\in \p\O_1}(1-P_\gamma)f|^2_{2,+} \dd \tau \Big]. \label{b_estimate}
\end{align}

\textbf{Step 3: estimate of $a(x)$.} 

First we choose test function as $\psi = \phi(x)\chi_0$. Direct computation yields
\begin{equation}\label{transport_a_t}
v\cdot \nabla \psi =  \sum_{i=1}^3 \p_i \phi \chi_i.
\end{equation}

LHS \eqref{weak_formula} becomes
\begin{align*}
   LHS &  =  \int_{\O} \{a(t)-a(t-\triangle) \} \phi \dd x.
\end{align*}

By \eqref{transport_a_t}, we have
\begin{equation}\label{I1_a}
I_1 = \int^t_{t-\triangle} \int_{\O} (\mathbf{b}\cdot \nabla \phi) \dd x \dd \tau,
\end{equation}
while $I_3=0$ and the contribution of $g$ in $I_4$ vanishes from the property $\mathcal{L}$ and \eqref{assumption_g}.

For the contribution of $\lambda f$ in $I_4$, we have
\begin{align}
    & \lambda \Big| \int_{\O\times \mathbb{R}^3} f \phi(x) \chi_0 \dd x \dd v \Big|  = \lambda \Big| \int_{\O\times \mathbb{R}^3}  \phi(x) a(x) \dd x \dd v \Big| \leq \lambda \Vert \phi\Vert_{L^2_x}^2 + \lambda \Vert a\Vert_{L^2_x}^2 . \label{I4_a_bdd}
\end{align}

For fixed $t$ we choose $\phi = \Phi_a$ such that 
\begin{equation}\label{test_a_t}
\begin{split}
    -\Delta \Phi_a  = \p_\tau a(t) \text{ in }\O,  \ \ \nabla\Phi_a \cdot n = 0 \text{ on }\p\O, \ \  \int_{\O} \Phi_a \dd x = 0 .
\end{split}
\end{equation}

Then the contribution of the boundary integral in \eqref{weak_formula} becomes
\begin{align}
    &I_2= \int_{\gamma} \Phi_a(x)\chi_0 f = \int_{\p\O} \int_{\gamma_+} \Phi_a(x) \chi_0 f + \int_{\p\O} \int_{\gamma_-} \Phi_a(x)\chi_0 f   \notag \\
    & = \int_{\p \O_1}\int_{\gamma_+} \Phi_a(x)\chi_0 f - \int_{\p\O_1} \int_{\gamma_+} \Phi_a(x)\chi_0 f. \label{I2_a}
\end{align}
Here we applied the change of variable $v\to v-2(n(x)\cdot v)n(x)$ in \eqref{cov_spec} to eliminate the contribution of $\p\O_2$ from the specular boundary condition. Then use the change of variable in \eqref{cov_diffuse} to obtain
\begin{equation}\label{I2_a_bdd}
\eqref{I2_a} \lesssim o(1)\Vert \Phi_a \Vert_{L^2(\p\O)}^2 + |\mathbf{1}_{x\in \p\O_1}(1-P_\gamma)f|^2_{2,+}.
\end{equation}

By \eqref{I1_a}, \eqref{I4_a_bdd} and \eqref{I2_a_bdd}, we apply the boundary condition in \eqref{test_a_t} and trace theorem to have
\begin{align*}
    &\int_{\O} |\nabla \Phi_a|^2 \dd x =  -\int_{\O} \Delta \Phi_a \Phi_a \dd x = \int_{\O} \Phi_a \p_\tau a \dd x   \\
    & \lesssim  \Vert \nabla \Phi_a\Vert_{L^2_x} \Vert \mathbf{b}\Vert_{L^2_x} + o(1)\Vert \Phi_a \Vert_{H^1_x}^2 + o(1)\Vert a\Vert_{L^2_x}^2 + |\mathbf{1}_{x\in \p\O_1}(1-P_\gamma) f|^2_{2,+}.
\end{align*}
Combining with the standard Poincare inequality, we conclude that
\begin{align}
  \Vert \Phi_a\Vert_{H^1_x}  &  \lesssim  o(1)\Vert a\Vert_{L^2_x} + \Vert \mathbf{b}\Vert_{L^2_x} + |\mathbf{1}_{x\in \p\O_1}(1-P_\gamma) f|_{2,+} . \label{Phi_a_bdd}
\end{align}

Next we use the weak formulation in \eqref{weak_formula_2}. Let $\phi_a$ be a solution of the following problem
\begin{equation}\label{phi_a}
-\Delta \phi_a = a(t) \text{ in }\O, \ \nabla\phi_a \cdot n = 0 \text{ on } \p\O.
\end{equation}

We choose the test function as
\begin{equation}\label{test_a}
\psi=\psi_a := \sum_{i=1}^3  \p_i \phi_a  v_i (|v|^2 - 10) \mu^{1/2} = \sum_{i=1}^3 \p_i \phi_a (\sqrt{10} \hat{B}_i - 5\chi_i). 
\end{equation}
Then direct computation yields
\begin{align*}
  -v\cdot \nabla_x \psi  & =-5 \Delta \phi_a  \chi_0 - \sum_{i,j=1}^3 \p^2_{ij} \phi_a (\mathbf{I}-\mathbf{P})(v_iv_j(|v|^2-10)\mu^{1/2}). 
\end{align*}
Thus LHS of~\eqref{weak_formula_2} becomes 
\begin{align}
  LHS  &  = 5 \int_s^t \Vert a\Vert_{L^2_x}^2\dd \tau - \sum_{i,j=1}^3 \int_s^t \int_{\O} \p_{ij}^2 \phi_a \langle v_iv_j (|v|^2-10)\mu^{1/2},(\mathbf{I}-\mathbf{P})f\rangle \dd x \dd \tau \notag\\
  & = 5 \int_s^t \Vert a\Vert_{L^2_x}^2\dd \tau + E_3, \label{LHS_a}
\end{align}
where, for any $\delta_3>0$,
\begin{align*}
  |E_3|  & \lesssim \delta_3 \int_s^t \Vert a\Vert_{L^2_x}^2\dd \tau  +  \frac{1}{\delta_3} \int_s^t \Vert \mu^{1/4}(\mathbf{I}-\P)f\Vert^2_{L^2} \dd \tau.
\end{align*}
For $J_1$ in~\eqref{weak_formula_2}, from $\Phi_a = \p_\tau \phi_a$ and~\eqref{Phi_a_bdd}, we have
\begin{align}
  |J_1|  &  \lesssim  \int_s^t \Vert \Phi_a \Vert^2_{H^1_x} \dd \tau +  \int_s^t \Vert \mathbf{b}\Vert_{L^2_x}^2 \dd \tau +  \int_s^t \Vert \mu^{1/4}(\mathbf{I}-\P)f\Vert^2_{L^2}  \dd \tau \notag\\
  & \lesssim  \int_s^t \Vert \mathbf{b}\Vert_{L^2_x}^2\dd \tau+ o(1)\int_{s}^t \Vert a\Vert_{L^2_x}^2 \dd \tau  +  \int_s^t \Vert \mu^{1/4}(\mathbf{I}-\P)f\Vert^2_{L^2}  \dd \tau + |\mathbf{1}_{x\in \p\O_1}(1-P_\gamma) f|^2_{2,+}. \label{J1_a_bdd}
\end{align}

Then we apply the boundary condition of $\phi_a$ and $f$ to compute $J_2$:
\begin{align*}
    &  \int_{\gamma} \psi f \dd \gamma = \int_{\gamma_+} \psi f \dd \gamma + \int_{\gamma_-} \psi f \dd \gamma.
\end{align*}

First we compute the specular part, i.e, $x\in \p\O_2$. Applying the change of variable $v\to v-2(n(x)\cdot v)n(x)$ in \eqref{cov_spec}, we have
\begin{align*}
    & \int_{\p\O_2} \Big[\int_{n(x)\cdot v>0} + \int_{n(x)\cdot v<0} \Big] (|v|^2-10) \mu^{1/2} (v\cdot \nabla_x \phi_a) (n\cdot v) f \dd v \dd S_x  \\
    & = 2\int_{\p\O_2}  \int_{n(x)\cdot v>0} (|v|^2 -10)\mu^{1/2} (n\cdot \nabla_x \phi_a) (n\cdot v)^2 f \dd v \dd S_x = 0.
\end{align*}
Here we have applied the specular boundary condition of $f$ to eliminate the contribution of $v$ in $v-2(n(x)\cdot v)n(x)$. In the last equality we used the boundary condition of $\phi_a$ in \eqref{phi_a}.

Then we compute the diffuse part, i.e, $x\in \p\O_1$. We have
\begin{align*}
    & \int_{\p\O_1} \Big[\int_{n(x)\cdot v>0} + \int_{n(x)\cdot v<0} \Big] (|v|^2-10) \mu^{1/2} (v\cdot \nabla_x \phi_a) (n\cdot v) f \dd v \dd S_x \notag \\
    &  = \int_{\p\O_1} \int_{n(x)\cdot v>0}   (|v|^2 - 10) \mu^{1/2} (v\cdot \nabla_x \phi_a) (n\cdot v) (f-P_\gamma f) \dd v \dd S_x  \notag\\
    & + 2\int_{\p\O_1} \int_{n(x)\cdot v>0} (|v|^2 - 10)\mu^{1/2} (n\cdot \nabla_x \phi_a) (n\cdot v)^2 P_\gamma f \dd v\dd S_x  \notag\\
    & = \int_{\p\O_1} \int_{n(x)\cdot v>0}   (|v|^2 - 10) \mu^{1/2} (v\cdot \nabla_x \phi_a) (n\cdot v) (f-P_\gamma f) \dd v \dd S_x  \notag\\
    & \lesssim \delta_3 \Vert \nabla \phi_a\Vert_{L^2(\p\O)}^2 + \frac{1}{\delta_3} |\mathbf{1}_{x\in \p\O_1} (1-P_\gamma)f|^2_{2,+}\lesssim \delta_3 \Vert a\Vert_{L^2_x}^2 + \frac{1}{\delta_3} |\mathbf{1}_{x\in \p\O_1} (1-P_\gamma)f|^2_{2,+}. 
\end{align*}
In the first equality we used the change of variable $v\to v-2(n(x)\cdot v)n(x)$ in \eqref{cov_diffuse}. In the second equality, the third line vanishes due to the boundary condition of $\phi_a$ in \eqref{phi_a}. In the last inequality, we used the standard elliptic estimate of \eqref{phi_a} with trace theorem: $\Vert \phi_a\Vert_{H^1(\p\O)}\lesssim \Vert \phi_a\Vert_{H^2_x}\lesssim \Vert a\Vert_{L^2_x}$.

We derive the estimate for $J_2$ as
\begin{equation}\label{J2_a_bdd}
|J_2| \lesssim \delta_3 \int_s^t \Vert a\Vert_{L^2_x}^2 \dd \tau + \frac{1}{\delta_3} \int_s^t |\mathbf{1}_{x\in \p\O_1} (1-P_\gamma )f|_{2,+}^2 \dd \tau.
\end{equation}

$J_3$ and the contribution of $g$ in $J_4$ are estimated similarly as \eqref{J3_c_bdd} and \eqref{J4_c_bdd}:
\begin{equation}\label{J3_a_bdd}
|J_3| \lesssim \delta_3 \int_s^t \Vert a\Vert_{L^2_x}^2 \dd \tau + \frac{1}{\delta_3} \int_s^t \Vert (\mathbf{I}-\mathbf{P})f \Vert^2_{L^2} \dd \tau,
\end{equation}
\begin{equation}\label{J4_a_bdd}
\Big|\int_s^t \int_{\O\times \mathbb{R}^3} g\psi\dd x \dd v \dd \tau\Big| \lesssim \delta_3 \int_s^t \Vert a\Vert_{L^2_x}^2 \dd \tau + \frac{1}{\delta_3} \int_s^t \Vert g\Vert_{L^2}^2 \dd \tau.
\end{equation}

For the contribution of $\lambda f$ in $J_4$, with the choice of $\psi$ in \eqref{test_a}, from the standard regularity estimate and $\lambda \ll 1$, we have
\begin{align}
    & \lambda \Big|\int_s^t \int_{\O\times \mathbb{R}^3} f\psi\dd x \dd v \dd \tau\Big|   \lesssim    \lambda \int_s^t \Vert \phi_a \Vert_{H^1_x}^2 \dd \tau   + \int_s^t \Vert \mathbf{b}\Vert_{L^2_x}^2 \dd \tau + \int_s^t \Vert \mu^{1/4} (\mathbf{I}-\mathbf{P})f \Vert_{L^2}^2 \dd \tau      \notag  \\
    & \lesssim o(1)\int_s^t  \Vert a \Vert_{L^2_x}^2 \dd \tau + \int_s^t \Vert \mathbf{b}\Vert_{L^2_x}^2 \dd \tau + \int_s^t \Vert \mu^{1/4} (\mathbf{I}-\mathbf{P})f \Vert_{L^2}^2 \dd \tau   . \label{J4_a_bdd_2} 
\end{align}

Collecting \eqref{LHS_a}, \eqref{J1_a_bdd}, \eqref{J2_a_bdd}, \eqref{J3_a_bdd}, \eqref{J4_a_bdd} and \eqref{J4_a_bdd_2}, we conclude the estimate $a$ as follows: for some $C_3>0$,
\begin{align}
  \int_s^t \Vert a\Vert_{L^2_x}^2 \dd \tau  & \leq C_3 \Big[ G_a(t)-G_a(s) + \int_s^t \Vert \mathbf{b}\Vert_{L^2_x}^2 \dd \tau  + \frac{1}{\delta_3} \int_s^t \Vert (\mathbf{I}-\mathbf{P})f\Vert_{L^2}^2 \dd \tau \notag\\
  & +  \frac{1}{\delta_3} \int_s^t \Vert g\Vert_{L^2}^2 \dd \tau + \frac{1}{\delta_3}\int_s^t |\mathbf{1}_{x\in \p\O_1} (1-P_\gamma )f|_{2,+}^2 \dd \tau\Big]. \label{a_bdd}
\end{align}

\textbf{Step 4: conclusion}  

We summarize \eqref{a_bdd}, \eqref{b_estimate} and \eqref{c_bdd}. We let $\delta_2=\sqrt{\delta_1}$, and multiply \eqref{b_estimate} by $\delta_1^{3/4}$ to have
\begin{align}
 \delta_1^{3/4} \int_s^t \Vert \mathbf{b}\Vert_{L^2_x}^2 \dd \tau  & \leq  C_2 \delta_1^{5/4} \int_s^t \Vert a\Vert_{L^2_x}^2 \dd \tau + C_2 \delta_1^{1/4} \int_s^t \Vert c\Vert^2_{L^2_x} \dd \tau   + C_2 \delta_1^{3/4} \Big[ G_b(t) - G_b(s)  \notag\\
  & +  \frac{1}{\sqrt{\delta_1}} \int_s^t \Vert (\mathbf{I}-\mathbf{P})f\Vert_{L^2}^2 \dd \tau + \frac{1}{\sqrt{\delta_1}}\int_s^t \Vert g\Vert_{L^2}^2 \dd \tau + \frac{1}{\sqrt{\delta_1}} \int_s^t |\mathbf{1}_{x\in \p\O_1}(1-P_\gamma)f|^2_{2,+} \dd \tau  \Big]. \label{b_estimate_delta}
\end{align}

Then we evaluate $\delta_1 \times \eqref{a_bdd} + \eqref{b_estimate_delta} + \eqref{c_bdd}   $ as
\begin{align*}
    &  \delta_1 \int_s^t \Vert a\Vert^2_{L^2_x} \dd \tau  +  \delta_1^{3/4} \int_s^t \Vert \mathbf{b}\Vert_{L^2_x}^2 \dd \tau +  \int_s^t \Vert c\Vert_{L^2_x}^2 \dd \tau \\
    & \leq   (C_3 \delta_1 + C_1 \delta_1) \int_s^t \Vert \mathbf{b}\Vert_{L^2_x}^2 \dd \tau + C_2 \delta_1^{5/4} \int_s^t \Vert a\Vert_{L^2_x}^2 \dd \tau + C_2 \delta_1^{1/4} \int_s^t \Vert c\Vert_{L^2_x}^2 \dd \tau \\
    & + C \Big[G_a(t)+G_b(t)+G_c(t) - G_a(s)-G_b(s)-G_c(s)  \Big] \\
    & + C \Big[\int_s^t \Vert g\Vert_{L^2}^2 \dd \tau + \int_s^t \Vert g\Vert_{L^2}^2 \dd \tau +  \int_s^t |\mathbf{1}_{x\in \p\O_1}(1-P_\gamma)f|^2_{2,+} \dd \tau  \Big].
\end{align*}
Here the constant $C$ in the last two lines depends on $C_1,C_2,C_3,\delta_1$. We choose small enough $\delta_1$ such that 
\begin{align*}
    &C_3\delta_1+C_1\delta_1 < \delta_1^{3/4}, \ \ C_2 \delta_1^{5/4} < \delta_1, \ \ C_2 \delta_1^{1/4}<1.
\end{align*}
Finally, we conclude the lemma with $|G(t)| = |G_a(t) + G_b(t) + G_c(t)| = |\int_{\O\times \mathbb{R}^3} (\psi_a + \psi_b +\psi_c) f(t) \dd x \dd v |  \lesssim \Vert f\Vert_{L^2}^2$ and $\mathbf{1}_{x\in \p\O_1} = \alpha(x)$.
\end{proof}

\begin{proof}[\textbf{Proof of Proposition \ref{prop:l2}}]
The well-posedness of~\eqref{intial_value_problem} is standard by constructing an approximating sequence. We refer the detailed proofs to Proposition 6.1 in~\cite{EGKM}.

In what follows we only prove the decay estimate~\eqref{l2_decay}. Multiplying~\eqref{intial_value_problem} with $e^{\lambda t}$ we get
\begin{equation}\label{f_lambda_t}
[\p_t  + v\cdot \nabla_x + \mathcal{L}](e^{\lambda t}f) =    \lambda e^{\lambda t}f + e^{\lambda t}g. 
\end{equation}
Applying Green's identity to~\eqref{f_lambda_t}, for some $\delta \ll 1$, we have
\begin{align}
&  \Vert e^{\lambda t}f(t)\Vert_{L^2}^2 + \int_0^t \Vert (\mathbf{I}-\mathbf{P})e^{\lambda s}f(s)\Vert_\nu^2  \dd s+ \int_0^t  |\mathbf{1}_{x\in \p\O_1}(1-P_\gamma)e^{\lambda s}f(s)|_{2,+}^2    \dd s \notag\\
& \lesssim  (\lambda+\delta) \int_0^t \Vert e^{\lambda s} f(s)\Vert_{L^2}^2 \dd s + \Vert f(0)\Vert_{L^2}^2 + C_\delta\int_0^t  \Vert e^{\lambda s} g(s)\Vert_{L^2}^2 \dd s. \label{green_lambda}
\end{align}
Furthermore, since $\mathbf{P}(e^{\lambda t}g)=0$, we can apply Lemma \ref{lemma:L2} to~\eqref{f_lambda_t} with $\lambda \ll 1$, and then we have
\begin{align}
  \int_0^t \Vert e^{\lambda s}\mathbf{P}f(s)\Vert_{L^2}^2  &  \lesssim G(t)-G(0) + \int_0^t \Vert e^{\lambda s}(\mathbf{I}-\mathbf{P})f(s)\Vert_{L^2}^2  \notag \\
  & + \int_0^t \Vert e^{\lambda s} g(s)\Vert_{L^2}^2 \dd s + \int_0^t |e^{\lambda s}\mathbf{1}_{x\in \p\O_1}    (1-P_\gamma) f(s)|_{2,+}^2   \dd s,   \label{pf_lambda}
\end{align}
where $|G(t)|\lesssim \Vert e^{\lambda t} f(t)\Vert_{L^2}^2$. Therefore, multiplying \eqref{pf_lambda} by a small constant $\e$ and adding the resultant to~\eqref{green_lambda}, we obtain
\begin{align*}
 & (1-C\e) \Vert e^{\lambda t}f(t)\Vert_{L^2}^2  +  \Big\{(1-C\e)\int_0^t \Vert e^{\lambda s}(\mathbf{I}-\mathbf{P})f(s)\Vert_\nu^2 \dd s + \e\int_0^t\Vert e^{\lambda s}\mathbf{P}f(s)\Vert_{L^2}^2  \dd s\Big\} \\
 &+ (1-C\e) \int_0^t |e^{\lambda s}\mathbf{1}_{x\in \p\O_1}    (1-P_\gamma) f(s)|_{2,+}^2   \dd s \\
 & \leq  C(\lambda+\delta) \int_0^t \Vert e^{\lambda s} f(s)\Vert_{L^2}^2 \dd s   + C\Vert f(0)\Vert_{L^2}^2   + C_\delta \int_0^t \Vert e^{\lambda s} g(s)\Vert_{L^2}^2 \dd s  .
\end{align*}
Since $\Vert e^{\lambda s}(\mathbf{I}-\mathbf{P})f(s)\Vert_{L^2}^2 + \Vert e^{\lambda s}\mathbf{P}f(s)\Vert_{L^2}^2 \lesssim \Vert f(s)\Vert_{L^2}^2$, we further obtain with $\e \ll 1$:
\begin{align*}
 &  \Vert e^{\lambda t}f(t)\Vert_{L^2}^2  +  \e \int_0^t \Vert e^{\lambda s}f(s)\Vert_{L^2}^2 \dd s \\
 & \leq  C(\lambda+\delta) \int_0^t \Vert e^{\lambda s} f(s)\Vert_{L^2}^2 \dd s   + C\Vert f(0)\Vert_{L^2}^2   + C_\delta \int_0^t \Vert e^{\lambda s} g(s)\Vert_{L^2}^2 \dd s  .
\end{align*}
Last we let $\delta+\lambda \ll 1$ be such that $C(\lambda+\delta) \leq \e$, then the above estimate gives the desired decay estimate \eqref{l2_decay}. We conclude the proof of Proposition \ref{prop:l2}.
\end{proof}

\ \\

\section{$L^\infty$ estimate}\label{sec:linfty}
In this section we are devoted to the proof of Theorem \ref{thm:linfty}. 

We start with the $L^\infty$ estimate of the linear problem~\eqref{linear_f} in the following proposition.

\begin{proposition}\label{prop:linfty}
There is $C>0$ such that the solution in Proposition \ref{prop:l2} satisfies
\begin{align*}
 \Vert wf(t)\Vert_\infty + |wf(t)|_\infty   &   \leq  Ce^{-\lambda t} \Big\{\Vert wf_0\Vert_\infty + \sup_{0\leq s\leq t} e^{\lambda s}\Big\Vert \frac{w}{\langle v\rangle} g(s) \Big\Vert_\infty \Big\},
\end{align*}
for any $t\geq 0$.
\end{proposition}

We use standard notations for the backward exit time and backward exit position:
\begin{equation*}
\begin{split}
\tb(x,v) :   &  = \sup\{s\geq 0, x-sv \in \O\} , \\
  \xb(x,v)  : & = x - \tb(x,v)v.
\end{split}
\end{equation*}

Recall the specular portion and diffuse portion defined in~\eqref{bdr_portion}. We denote $t_0=T_0$, a fixed starting time. First we define the stochastic cycle with respect to the specular reflection:
\begin{definition}\label{def:sto_cycle_spec}
When $\xb(x,v) \in \p \O_2$, the specular portion, we define
\begin{equation}\label{spec_cycle_1}
\begin{split}
    \bar{x}_{0,1}   =   \xb(x,v), \ \    \bar{t}_{0,1} & = t_0 - \tb(x,v), \ \ \bar{v}_{0,1} = R_{\bar{x}_{0,1}}v. 
\end{split}
\end{equation}
Otherwise, the backward exit position belongs to $\p \O_1$, and \eqref{spec_cycle_1} is not defined. 

Inductively, if $\bar{x}_{0,j}$, $\bar{t}_{0,j}$ and $\bar{v}_{0,j}$ are defined, and $\xb(\bar{x}_{0,j}, \bar{v}_{0,j}) \in \p\O_2$, we define the specular cycle as
\begin{equation}\label{spec_cycle_j}
    \bar{x}_{0,j+1} = \xb(\bar{x}_{0,j}, \bar{v}_{0,j}),\ \  \bar{t}_{0,j+1} = \bar{t}_{0,j} - \tb(\bar{x}_{0,j}, \bar{v}_{0,j}), \ \ \bar{v}_{0,j+1} = R_{\bar{x}_{0,j}}\bar{v}_{0,j}.
\end{equation}

Then we define the maximum index that corresponds to the iteration of the specular reflection:
\begin{equation}\label{maximum_spe}
M_0 = \sup\{j\geq 0:\bar{x}_{0,j} \text{ is defined}\}.
\end{equation}
Here we note that in the case of $v_2=0, v_3=0$, we have $M_0=\infty$; in the case of $M_0=0$, we have $\xb(x,v)\in \p\O_1$.

Conventionally, we define
\begin{equation}\label{spec_cycle_0}
\bar{x}_{0,0} = x, \ \ \bar{t}_{0,0} = t_0, \ \ \bar{v}_{0,0} = v.
\end{equation}

\end{definition}

Next we define the stochastic cycle with respect to the diffuse reflection:
\begin{definition}\label{def:sto_cycle_diffuse}
Continuing from Definition \ref{def:sto_cycle_spec}, when $M_0 < \infty$, we define
\begin{equation*}
x_1  = \xb(\bar{x}_{0,M_0}, \bar{v}_{0,M_0}), \ \ t_1 = \bar{t}_{0,M_0} - \tb(\bar{x}_{0,M_0}, \bar{v}_{0,M_0}), \ \ v_1 \in \mathcal{V}_1 = \{v_1\in \mathbb{R}^3: n(x_1)\cdot v_1>0\}.
\end{equation*}
The definition of $M_0$ in~\eqref{maximum_spe} implies that $x_1\in \p\O_1$, the diffuse portion.

Similar to~\eqref{spec_cycle_1}, when $\xb(x_1,v_1) \in \p\O_2$, we define
\begin{equation*}
\bar{x}_{1,1} = \xb(x_1,v_1), \ \ \bar{t}_{1,1} = t_1 - \tb(x_1,v_1),  \ \ \bar{v}_{1,1} = R_{\bar{x}_{1,1}} v_1.
\end{equation*}

Inductively, if $\bar{x}_{i,j}$, $\bar{t}_{i,j}$ and $\bar{v}_{i,j}$ are defined, and $\xb(\bar{x}_{i,j},\bar{v}_{i,j})\in \p\O_2$, similar to the specular cycle in~\eqref{spec_cycle_j}, we further define
\begin{equation*}
    \bar{x}_{i,j+1} = \xb(\bar{x}_{i,j}, \bar{v}_{i,j}),\ \  \bar{t}_{i,j+1} = \bar{t}_{i,j} - \tb(\bar{x}_{i,j}, \bar{v}_{i,j}), \ \ \bar{v}_{i,j+1} = R_{\bar{x}_{i,j}}\bar{v}_{i,j},
\end{equation*}
and
\begin{equation*}
M_i = \sup\{j\geq 0: \bar{x}_{i,j} \text{ is defined}\}.
\end{equation*}

Then inductively, we define the diffuse cycle as
\begin{equation*}
\begin{split}
& x_{i+1}  = \xb(\bar{x}_{i,M_i}, \bar{v}_{i,M_i}), \ \ t_{i+1} = \bar{t}_{i,M_i} - \tb(\bar{x}_{i,M_i}, \bar{v}_{i,M_i}), \\
&v_{i+1} \in \mathcal{V}_{i+1} = \{v_{i+1}\in \mathbb{R}^3: n(x_{i+1})\cdot v_{i+1}>0\}.    
\end{split}
\end{equation*}
Conventionally, we denote
\begin{equation*}
\bar{x}_{i,0} = x_i, \ \ \bar{t}_{i,0} = t_i, \ \ \bar{v}_{i,0} = v_i,
\end{equation*}
and
\begin{equation*}
x_0 = x, \ \ v_0 = v, \ \ t_0= T_0.
\end{equation*}

For ease of notation, we denote $\tilde{v}_{i} = \bar{v}_{i,M_i}$, where $|v_{i}|^2 = |\tilde{v}_{i}|^2$.

For the characteristic between two diffuse boundary interactions, conventionally we denote $\bar{t}_{i,M_i+1} = t_{i+1}$, with $t_{i+1}\leq s\leq t_i$, we denote
\begin{equation*}
\begin{split}
 X_i(s) =  X(x_i,v_i;s,t_i)  = &  \sum_{k=0}^{M_i} \mathbf{1}_{ \bar{t}_{i,k+1}\leq s \leq \bar{t}_{i,k}} \big\{\bar{x}_{i,k} - (\bar{t}_{i,k} - s)\bar{v}_{i,k} \big\}, \\
 V_i(s) = V(x_i,v_i;s,t_i) & =\sum_{k=0}^{M_i} \mathbf{1}_{ \bar{t}_{i,k+1}\leq s \leq \bar{t}_{i,k}} \bar{v}_{i,k}.
\end{split}
\end{equation*}

\end{definition}

With all stochastic cycles defined, suppose $f$ satisfies the linear equation \eqref{linear_f}. Since $w(v) = w(V_0(s))$ from the specular reflection, we apply the method of characteristic to get
\begin{align}
  &w(v) f(T_0,x,v)  \notag \\
  & =  \mathbf{1}_{t_1\leq 0} w(V_0(v)) e^{-\nu T_0} f(0,X_0(0), V_0(0)) \label{initial}\\
  & + \mathbf{1}_{t_1 \leq 0} \int_0^{T_0} e^{-\nu(T_0-s)} \int_{\mathbb{R}^3} w(V_0(s))\mathbf{k}(V_0(s),u) f(s,X_0(s),u) \dd u \dd s \label{K_0}\\
  & +\mathbf{1}_{t_1>0} \int^{T_0}_{t_1} e^{-\nu (T_0 - s)} \int_{\mathbb{R}^3} w(V_0(s)) \mathbf{k}(V_0(s),u) f(s,X_0(s),u) \dd u \dd s  \label{K_1} \\
  &  + \mathbf{1}_{t_1 \leq 0} \int_0^{T_0} e^{-\nu(T_0-s)}  w(V_0(s))  g(s,X_0(s),V_0(s)) \dd s  \label{g_0} \\
  &  +\mathbf{1}_{t_1>0} \int^{T_0}_{t_1} e^{-\nu (T_0 - s)} w(V_0(s)) g(s,X_0(s),V_0(s)) \dd s   \label{g_1} \\
  & + \mathbf{1}_{t_1>0} e^{-\nu (T_0 - t_1)} w(\tilde{v}_0)f(t_1,x_1,\tilde{v}_0), \label{f_bdr}
\end{align}
where the contribution of the boundary is bounded as
\begin{align}
   & |\eqref{f_bdr}|\leq   e^{-\nu (T_0-t_1)} w(\tilde{v}_0)  \notag\\
   & \times \int_{\prod_{j=1}^{k-1}\mathcal{V}_j} \bigg\{ \sum_{i=1}^{k-1}\mathbf{1}_{t_{i+1}\leq 0 < t_i} e^{-\nu t_1} w(V_i(0))|f(0, X_i(0), V_i(0))| \dd \Sigma_i     \label{bdr_initial}\\
  & + \mathbf{1}_{t_{k}>0} e^{-\nu (t_1-t_{k})} w(\tilde{v}_{k-1})|f(t_{k},x_k,\tilde{v}_{k-1})| \dd \Sigma_{k-1} \label{bdr_tk}\\
  & + \sum_{i=1}^{k-1} \mathbf{1}_{t_{i+1}\leq 0 < t_i}  \int^{t_i}_0 e^{-\nu (t_1-s)} \int_{\mathbb{R}^3} w(V_i(s))\mathbf{k}(V_i(s),u) f(s,X_i(s), u) \dd s \dd \Sigma_i   \label{bdr_K_0} \\
  & + \sum_{i=1}^{k-1} \mathbf{1}_{t_{i+1}>0} \int_{t_{i+1}}^{t_i} e^{-\nu (t_1-s)} \int_{\mathbb{R}^3} w(V_i(s))\mathbf{k}(V_i(s),u) f(s,X_i(s), u) \dd s \dd \Sigma_i   \label{bdr_K_i} \\
  &  + \sum_{i=1}^{k-1} \mathbf{1}_{t_{i+1}\leq 0 < t_i}  \int^{t_i}_0 e^{-\nu (t_1-s)} w(V_i(s)) g(s,X_i(s),V_i(s)) \dd \Sigma_i     \label{bdr_g_0}   \\
  &  + \sum_{i=1}^{k-1} \mathbf{1}_{t_{i+1}>0} \int_{t_{i+1}}^{t_i} e^{-\nu (t_1-s)} w(V_i(s)) g(s,X_i(s),V_i(s)) \dd \Sigma_i  \bigg\}.  \label{bdr_g_i}  
\end{align}
Here $\dd \Sigma_i$ is defined as
\begin{equation*}
\dd \Sigma_i = \Big\{\prod_{j=i+1}^{k-1} \dd \sigma_j \Big\}  \times  \Big\{  \frac{1}{w(v_i)\sqrt{\mu(v_i)}}   \dd \sigma_i \Big\} \times \Big\{\prod_{j=1}^{i-1}\dd \sigma_j \Big\} ,
\end{equation*}
where $\dd \sigma_i$ is a probability measure in $\mathcal{V}_i$ given by
\begin{equation}\label{sigma_i}
\dd \sigma_i =  \sqrt{2\pi}\mu(v_i) (n(x_i)\cdot v_i)\dd v_i.
\end{equation}
Here we note that we have $\mu(v_i)$ in~\eqref{sigma_i} due to the specular reflection property $|v_i|^2 = |\tilde{v}_i|^2$.

We begin with an estimate of~\eqref{bdr_tk}, which corresponds to the scenario that the backward trajectory interacts with the diffuse boundary portion for a large amount of times.

\begin{lemma}\label{lemma:tk}
For $T_0>0$ sufficiently large, there exists constants $C_1,C_2>0$ independent of $T_0$ such that for $k = C_1 T_0^{5/4}$, and $(t_0,x_0,v_0) = (t,x,v)\in [0,T_0]\times \bar{\O}\times \mathbb{R}^3$,
\begin{equation*}
\int_{\prod_{j=1}^{k-1} \mathcal{V}_j}\mathbf{1}_{t_k>0}  \prod_{j=1}^{k-1} \dd \sigma_j \leq \Big( \frac{1}{2}\Big)^{C_2 T_0^{5/4}}.
\end{equation*}

\end{lemma}

\begin{proof}
We choose $\delta\ll 1$ to be small, and consider the non-grazing sets as 
\begin{equation*}
\mathcal{V}_j^\delta    =  \{v_j\in \mathcal{V}_j : n(x_j)\cdot v_j \geq \delta \text{ and }|v_j|\leq \frac{1}{\delta}\}.
\end{equation*}
Then the integration over the grazing set induces a small magnitude number:
\begin{equation*}
\int_{\mathcal{V}_j \backslash \mathcal{V}_j^\delta} \dd \sigma_j \leq C\delta.
\end{equation*}
For the non-grazing set, we can obtain a lower bound of the backward exit time, that is
\begin{equation*}
t_i-t_{i+1}\geq t_i - \bar{t}_{i,1} \geq \frac{\delta^3}{C_\O}.
\end{equation*}
Hence there can be at most $\big[ \frac{C_\O T_0}{\delta^3} \big] + 1$ number of $v_j\in \mathcal{V}_j^\delta$ for $1\leq j\leq k-1$. Then we follow the argument in~\cite{G} to conclude the lemma.
\end{proof}

To prove Proposition \ref{prop:linfty}, we need to estimate every term in the characteristic formula~\eqref{initial} - \eqref{f_bdr}. First we estimate the boundary term \eqref{f_bdr} in the following lemma.

\begin{lemma}\label{lemma:bdr}
For the boundary term~\eqref{f_bdr}, we have
\begin{equation*}
\begin{split}
 w(\tilde{v}_0)|f(t_1,x_1,\tilde{v}_0)| \leq   & C(\theta) e^{-\nu t_1} \Vert w f_0\Vert_\infty+ o(1)e^{-\lambda t_1}\sup_{0\leq s\leq T_0} \Vert e^{\lambda s}w f(s)\Vert_\infty  \\
    &+   C(T_0) e^{-\lambda t_1}  \sup_{0\leq s\leq T_0}  \Big\Vert e^{\lambda s} \frac{w}{\langle v\rangle} g(s) \Big\Vert_\infty  +  C(T_0) e^{-\lambda t_1} \sup_{0\leq s\leq T_0} e^{\lambda s}\Vert f(s)\Vert_{L^2} .
\end{split}
\end{equation*}
\end{lemma}

\begin{proof}
Since $\dd \sigma_i$ in~\eqref{sigma_i} is a probability measure, \eqref{bdr_initial} is directly bounded as
\begin{equation}\label{bdr_initial_bdd}
\eqref{bdr_initial}\leq C(\theta) e^{-\nu t_1}   \Vert w f_0\Vert_\infty,
\end{equation}
where the constant $C(\theta)$ comes from
\[C(\theta) = \int_{\mathcal{V}_i}   |n(x_i)\cdot v_i| \sqrt{\mu(v_i)} w^{-1}(v_i) \dd v_i.\]

For~\eqref{bdr_tk}, with $\lambda \ll \nu$ and $k=C_1 T_0^{5/4}$, we apply Lemma \ref{lemma:tk} to have
\begin{equation}\label{bdr_tk_bdd}
\begin{split}
 \eqref{bdr_tk}   &  \leq  e^{-\lambda t_1}\int_{\prod_{j=1}^{k-1} \mathcal{V}_j} \mathbf{1}_{t_k>0}  |e^{\lambda t_k} w(\tilde{v}_{k-1}) f(t_k,x_k,\tilde{v}_{k-1})| \dd \Sigma_{k-1} \\
 &\leq o(1)e^{-\lambda t_1} \sup_{0\leq s\leq T_0}\Vert e^{\lambda s } w f(s)\Vert_\infty.
\end{split}
\end{equation}

For \eqref{bdr_g_0} and \eqref{bdr_g_i} they are directly bounded as
\begin{equation}\label{bdr_g_bdd}
\begin{split}
|\eqref{bdr_g_0} + \eqref{bdr_g_i}| & \leq      Cke^{-\lambda t_1} \sup_{0\leq s\leq T_0} e^{\lambda s}\Big\Vert \frac{w}{\langle v\rangle} g(s) \Big\Vert_\infty   \int_0^{T_0} e^{-\nu(V_i(s))(T_0-s)/2} \langle V_i(s) \rangle \dd s \\ 
& \leq  Cke^{-\lambda t_1} \sup_{0\leq s\leq T_0} e^{\lambda s}\Big\Vert \frac{w}{\langle v\rangle} g(s) \Big\Vert_\infty .
\end{split}
\end{equation}

Then we estimate \eqref{bdr_K_i}. We denote $\mathbf{k}_\theta(v,u) := \mathbf{k}(v,u)\frac{e^{\theta|v|^2}}{e^{\theta |u|^2}}.$ We focus on estimating
\begin{equation}\label{iteration_i}
\begin{split}
    & \int_{\prod_{j=1}^i \mathcal{V}_j}\mathbf{1}_{t_{i+1}>0 }  \prod_{j=1}^i\dd \sigma_j \mu^{-1/2}(v_i) w^{-1}(v_i)   \\
    & \times     \int_{t_{i+1}}^{t_i} e^{-\nu(t_1 -s)} \int_{\mathbb{R}^3} \mathbf{k}_\theta(V_i(s),u) w(u)f(s,X_i(s),u) \dd s .
\end{split}
\end{equation}

First we decompose the $\dd s$ integral into $\mathbf{1}_{s\geq t_i-\delta} + \mathbf{1}_{s< t_i-\delta}$. By \eqref{k_theta} in Lemma \ref{lemma:k_theta}, the contribution of the first term reads
\begin{align}
 \eqref{iteration_i} \mathbf{1}_{s\geq t_i-\delta}   &  \leq \int_{\prod_{j=1}^{i}\mathcal{V}_j} \prod_{j=1}^{i}\dd \sigma_j \mu^{-1/2}(v_i) w^{-1}(v_i) \notag\\ 
 & \times \int^{t_i}_{\max\{t_{i+1},t_i-\delta\}} e^{-\nu (t_1-s)} \int_{\mathbb{R}^3} \mathbf{k}_\theta(V_i(s),u) w(u) f(s,X_i(s),u) \dd s  \notag\\
 &\leq o(1)e^{-\lambda t_1} \sup_{0\leq s\leq T_0}\Vert e^{\lambda s} w f(s)\Vert_\infty . \label{bdr_s_small_bdd}
\end{align}

Next we decompose the $v_i$ integral into $\mathbf{1}_{|v_i|\geq N} + \mathbf{1}_{|v_i|<N}$. By \eqref{k_theta} in Lemma \ref{lemma:k_theta}, the contribution of the first term reads
\begin{align}
    & \eqref{iteration_i} \mathbf{1}_{|v_i|\geq N}  \leq \int_{\prod_{j=1}^{i-1}\mathcal{V}_j} \prod_{j=1}^{i-1}\dd \sigma_j  \int_{\mathcal{V}_i} \mathbf{1}_{|v_i|\geq N} \sqrt{\mu(v_i)} w^{-1}(v_i)|n(x_i)\cdot v_i| \dd v_i   \notag \\
    &\ \ \ \ \ \ \ \ \ \ \ \ \ \ \ \ \  \times \int_{t_{i+1}}^{t_i} e^{-\nu(t_1 -s)} \int_{\mathbb{R}^3} \mathbf{k}_\theta(V_i(s),u) w(u) f(s,X_i(s),u) \dd s \notag\\
    & \leq o(1) e^{-\lambda t_1} \sup_{0\leq s\leq T_0}\Vert e^{\lambda s} w f(s)\Vert_\infty.   \label{bdr_v_large_bdd}
\end{align}

Then we decompose the $u$ integral into $\mathbf{1}_{|u|\geq N \text{ or } |V_{i}(s)-u|\leq \frac{1}{N}} + \mathbf{1}_{|u|<N, \ |V_i(s)-u|> \frac{1}{N}}$. By \eqref{K_N_small} in Lemma \ref{lemma:k_theta}, the contribution of the first term reads
\begin{align}
    & \eqref{iteration_i} \mathbf{1}_{|u|\geq N \text{ or } |V_i(s)-u|\leq \frac{1}{N}}  \leq \int_{\prod_{j=1}^{i}\mathcal{V}_j} \prod_{j=1}^{i}\dd \sigma_j  \mu^{-1/2}(v_i) w^{-1}(v_i)  \notag\\
    &\ \ \times  \int_{t_{i+1}}^{t_i} e^{-\nu(t_1 -s)} \int_{\mathbb{R}^3}  \mathbf{1}_{|u|\geq N \text{ or } |V_i(s)-u|\leq \frac{1}{N}} \mathbf{k}_\theta(V_i(s),u) w(u) f(s,X_i(s),u) \dd s \notag\\
    & \leq o(1) e^{-\lambda t_1} \sup_{0\leq s \leq T_0}\Vert e^{\lambda s}w f(s)\Vert_\infty.   \label{bdr_u_large_bdd}
\end{align}


Now we consider the intersection of all other cases, where we have $|v_i|\leq N, \ s<t_i-\delta$, and $|u|<N, \ |V_i(s)-u|>\frac{1}{N}$. The conditions of $v_i$ and $u$ imply that $\mathbf{k}(V_i(s),u)\leq C_N$ from \eqref{k_theta_bdd} in Lemma \ref{lemma:k_theta}. 

For fixed $x_i$ and $s\in [t_{i+1},t_i]$, for $v_i\in \mathcal{V}_i$ such that $|v_i|\leq N$ and $X_i(s)\in \bar{\O}$, we denote 
\[\tilde{\O} = \{x\in \mathbb{R}^3| x=x_i-(t_i-s)v_i, \ |v_i|\leq N, X_i(s)\in \bar{\O}\}.\] Then we define a map 
\[\mathcal{M}:  x_i - (t_i-s)v_i \in \tilde{\O} \to X_i(s) \in \O .\] 
Since $|v_i|\leq N$, for any $x\in \tilde{\O}$, we have $x \cdot e_1 \in [-L-NT_0,L+NT_0]$. We denote \[\tilde{\O}_0 = \{x\in \tilde{\O}| x \cdot e_1 \in [-L,L]\},\] 
and for $i>0$, we denote 
\begin{equation}\label{omega_i_+}
\tilde{\O}_i = \{x\in \tilde{\O}| x \cdot e_1 \in [iL,(i+2)L], \ i>0\}.    
\end{equation}
Similarly, for $i<0$ we denote 
\begin{equation}\label{omega_i_-}
\tilde{\O}_i = \{x\in \tilde{\O}|x \cdot e_1 \in [(i-2)L,iL], \ i<0\}.    
\end{equation}
Then we have 
\[\tilde{\O} = \bigcup_{i=-1-NT_0/L}^{1+NT_0/L} \tilde{\O}_i.\]

If $x\in \tilde{\O}_0$, then $\mathcal{M}(x) = x$. If $x\in \tilde{\O}_i$ with even $i$, then $\mathcal{M}(x) = x-iL$. If $x\in \tilde{\O}_i$ with odd $i$, then $\mathcal{M}(x) = -(x-iL)$.

Then we compute as
\begin{align}
  & \eqref{iteration_i} \mathbf{1}_{|v_i|\leq N} \mathbf{1}_{s<t_i-\delta} \mathbf{1}_{|u|<N, \ |V_i(s)-u|>\frac{1}{N}}  \notag\\
  &\leq   C_N\int_{\prod_{j=1}^{i-1}\mathcal{V}_j} \prod_{j=1}^{i-1}\dd \sigma_j \int_{\mathcal{V}_i} \mathbf{1}_{|v_i|< N} \sqrt{\mu(v_i)} |n(x_i)\cdot v_i| \dd v_i   \notag\\ 
  &\times \int_{t_{i+1}}^{t_i-\delta}  e^{-\nu(t_1 -s)} \int_{|u|\leq N}  f(s,\mathcal{M}(x_i-(t_i-s)v_i),u) \dd u\dd s  \notag \\
  & \leq \frac{C_N}{\delta^3}\int_{\prod_{j=1}^{i-1}\mathcal{V}_j} \prod_{j=1}^{i-1}\dd \sigma_j  \int^{t_i-\delta}_{0}  e^{-\nu(t_1-s)}     \int_{|u|\leq N} \int_{\tilde{\O}} f(s,\mathcal{M}(y),u) \dd u \dd y \dd s.       \label{other_case}
\end{align}
In the last line we have applied the change of variable $v_i \to y = x_i-(t_i-s)v_i \in \tilde{\O}$ with Jacobian 
\[\Big|\det\Big(\frac{\p x_i-(t_i-s)v_i}{\p v_i} \Big) \Big| = (t_i-s)^3 \geq \delta^3 .\]

Then we proceed the computation as
\begin{align*}
  \eqref{other_case}  & \leq \frac{C_N}{\delta^3} \int_{\prod_{j=1}^{i-1} \mathcal{V}_j} \prod_{j=1}^{i-1} \dd \sigma_j \int_0^{t_i-\delta} e^{-\nu(t_1-s)} \int_{|u|\leq N} \sum_{i=-1-NT_0/L}^{1+NT_0/L} \int_{\tilde{\O}_i} f(s,\mathcal{M}(y),u) \dd u \dd y \dd s   \\ 
  & \leq \frac{C_N}{\delta^3} \int_{\prod_{j=1}^{i-1} \mathcal{V}_j} \prod_{j=1}^{i-1} \dd \sigma_j \int_0^{t_i-\delta} e^{-\nu(t_1-s)} \int_{|u|\leq N} \sum_{i=-1-NT_0/L}^{1+NT_0/L} \int_{\O} f(s,x,u) \dd u \dd x \dd s.
\end{align*}
In the second line we have applied the change of variable $y \in \tilde{\O}_i \to x=\mathcal{M}(y)\in \O$. When $i$ is even, the Jacobian is $|\det(\frac{\p [x-iL]}{\p x})| =1$. When $i$ is odd, the Jacobian is $|\det(\frac{\p [-(x-iL)]}{\p x})|=1$.

Then we apply the H\"older inequality to have
\begin{align}
  \eqref{other_case}    & \leq  C_{N,\delta,T_0,\O}   \int_{\prod_{j=1}^{i-1}\mathcal{V}_j} \prod_{j=1}^{i-1}\dd \sigma_j \times    \int_{0}^{t_1} e^{-\nu(t_1-s)} \Vert f(s)\Vert_{L^2} \dd s \notag\\
  & \leq  C_{N,\delta,T_0,\O} e^{-\lambda t_1} \sup_{0\leq s\leq T_0} e^{\lambda s} \Vert f(s)\Vert_{L^2}  . \label{other_case_bdd}
\end{align}

Collecting \eqref{bdr_s_small_bdd}, \eqref{bdr_v_large_bdd}, \eqref{bdr_u_large_bdd} and \eqref{other_case_bdd}, we conclude that
\begin{equation}\label{bdr_K_i_bdd}
\eqref{bdr_K_i} \lesssim o(1) e^{-\lambda t_1} \sup_{0\leq s\leq T_0} \Vert e^{\lambda s}w f(s)\Vert_\infty + C_{N,\delta,k,T_0,\O} e^{-\lambda t_1} \sup_{0\leq s\leq T_0} e^{\lambda s} \Vert f(s)\Vert_{L^2}. 
\end{equation}

By the same computation, we have the same bound for \eqref{bdr_K_0}:
\begin{equation}\label{bdr_K_0_bdd}
\eqref{bdr_K_0} \lesssim o(1) e^{-\lambda t_1} \sup_{0\leq s\leq T_0} \Vert e^{\lambda s}w f(s)\Vert_\infty + C_{N,\delta,k,T_0,\O} e^{-\lambda t_1} \sup_{0\leq s\leq T_0} e^{\lambda s} \Vert f(s)\Vert_{L^2}.  
\end{equation}

Summarizing \eqref{bdr_initial_bdd}, \eqref{bdr_tk_bdd}, \eqref{bdr_g_bdd}, \eqref{bdr_K_0_bdd} and \eqref{bdr_K_i_bdd}, we conclude the lemma.
\end{proof}

Now we are in position to prove Proposition \ref{prop:linfty}.

\begin{proof}[\textbf{Proof of Proposition \ref{prop:linfty}}]
First of all, \eqref{initial}, \eqref{g_0} and \eqref{g_1} are bounded as
\begin{equation}\label{initial_g_bdd}
|\eqref{initial}| + |\eqref{g_0}| + |\eqref{g_1}| \leq e^{-\nu_0 T_0} \Vert w f_0\Vert_\infty + Ce^{- \lambda T_0}\sup_{0\leq s\leq T_0}\Big\Vert e^{\lambda s} \frac{w}{\langle v\rangle} g(s) \Big\Vert_\infty.
\end{equation}
Moreover, \eqref{f_bdr} is bounded by Lemma \ref{lemma:bdr} as
\begin{equation}\label{f_bdr_bdd}
\begin{split}
 \eqref{f_bdr} \leq   &  C(\theta) e^{-\nu_0 T_0} \Vert w f_0\Vert_\infty + o(1)e^{-\lambda T_0}\sup_{0\leq s\leq T_0}\Vert e^{\lambda s} w f(s)\Vert_\infty  \\
    &  + C(T_0) \big[e^{-\lambda T_0} \Big\Vert e^{\lambda s} \frac{w}{\langle v\rangle} g(s) \Big\Vert_\infty +  e^{-\lambda T_0}\sup_{0\leq s\leq T_0} \Vert f(s)\Vert_{L^2} \big].
\end{split}
\end{equation}

Then we focus on~\eqref{K_1}. We expand $f(s,X_0(s),u)$ using the the characteristic \eqref{initial} - \eqref{f_bdr} again along $u$. 

For simplicity, we include a superscript $u$ to all notations in Definition \ref{def:sto_cycle_spec} and Definition \ref{def:sto_cycle_diffuse}, where the trajectory starts from $t_0^u = s$, $\bar{x}_{0,0}^u = X_0(s)$, and $\bar{v}_{i,0}^u = u$, corresponding to \eqref{spec_cycle_0}.

Then we have
\begin{align}
    & \eqref{K_1} = \mathbf{1}_{t_1>0} \int^{T_0}_{t_1} \dd s e^{-\nu (T_0 -s)}  \int_{\mathbb{R}^3} \dd u \frac{w(V_0(s))}{w(u)}\mathbf{k}(V_0(s),u) \notag \\
    &\times \Big\{ \mathbf{1}_{t_1^u \leq 0} e^{-\nu(u) s} w(V_0^u(0))f(0,X_0^u(0),V_0^u(0))       \label{u_initial} \\
    &  +\mathbf{1}_{t_1^u\leq 0} \int_0^{s}  e^{-\nu(u)(s-s')} \dd s'\int_{\mathbb{R}^3} w(V_0^u(s'))\mathbf{k}(V_0^u(s'),u') f(s',X_0^u(s'), u')    \dd u' \label{u_K0} \\
    &  +\mathbf{1}_{t_1^u> 0} \int_{t_1^u}^{s}  e^{-\nu(u)(s-s')} \dd s'\int_{\mathbb{R}^3} w(V_0^u(s')) \mathbf{k}(V_0^u(s'),u') f(s',X_0^u(s'), u')  \dd u' \label{u_K1}\\
    & + \mathbf{1}_{t_1^u\leq 0} \int_{0}^{s} e^{-\nu(u)(s-s')} w(V_0^u(s')) g(s',X_0^u(s'),V_0^u(s')) \dd s' \label{u_g0} \\
    & + \mathbf{1}_{t_1^u> 0} \int_{t_1^u}^{s} e^{-\nu(u)(s-s')} w(V_0^u(s')) g(s',X_0^u(s'),V_0^u(s')) \dd s' \label{u_g1} \\
    & + \mathbf{1}_{t_1^u > 0} e^{-\nu(u)(s-t_1^u)} w(\tilde{v}_0^u) f(t_1^u,x_1^u,\tilde{v}_0^u) \Big\}.   \label{u_bdr}
\end{align}

The contribution of \eqref{u_initial} in \eqref{K_1} is bounded by
\begin{align}
    & \int^{T_0}_{t_1} \dd s e^{-\nu_0 (T_0-s)} \int_{\mathbb{R}^3} \dd u \mathbf{k}_\theta(V_0(s),u) e^{-\nu_0 s} \Vert w f_0\Vert_\infty    \notag \\
    &\leq C(\theta)\int_{t_1}^{T_0} \dd s e^{-\nu(T_0-s)/2} e^{-\frac{\nu}{2}T_0} \Vert wf_0\Vert_\infty  \leq C(\theta)e^{-\frac{\nu_0 T_0}{2}} \Vert w f_0\Vert_\infty . \label{u_initial_bdd}
\end{align}
In the second line we have used Lemma \ref{lemma:k_theta}.

The contribution of \eqref{u_g0} and \eqref{u_g1} in \eqref{K_1} are bounded by
\begin{align}
    & \int_{t_1}^{T_0} \dd s e^{-\nu (T_0-s)} \int_{\mathbb{R}^3} \dd u  \mathbf{k}_\theta(V_0(s),u) \int^s_0 \dd s' e^{-\nu(u)(s-s')} \langle V_0^u(s') \rangle e^{-\lambda s'} \sup_{0\leq s\leq T_0}\Big\Vert  e^{\lambda s}\frac{w}{\langle v\rangle }g(s)\Big\Vert_\infty  \notag  \\
    & \leq C \sup_{0\leq s\leq T_0}\Big\Vert e^{\lambda s} \frac{w}{\langle v\rangle }g(s)\Big\Vert_\infty \int_{t_1}^{T_0} \dd s e^{-\nu (T_0-s)} e^{-\lambda s}\int_{\mathbb{R}^3} \dd u  \mathbf{k}_\theta(V_0(s),u) \notag \\
    &\leq C \sup_{0\leq s\leq T_0}\Big\Vert  e^{\lambda s}\frac{w}{\langle v\rangle }g(s)\Big\Vert_\infty. \label{u_g_bdd}
\end{align}
In the second line we have used \eqref{nu_bdd} and Lemma \ref{lemma:k_theta}.

The contribution of the boundary term in~\eqref{u_bdr} can be bounded by Lemma \ref{lemma:bdr}:
\begin{align}
  \eqref{u_bdr}  & \leq \int_{t_1}^{T_0} \dd s e^{-\nu(T_0-s)} \int_{\mathbb{R}^3} \dd u \mathbf{k}_\theta(V_0(s),u) e^{-\nu (s-t_1^u)} \notag \\
 & \times \Big\{ e^{-\nu_0 t_1^u}  C(\theta)\Vert w f_0\Vert_\infty + o(1)e^{-\lambda t_1^u}\sup_{0\leq s\leq T_0}\Vert e^{\lambda s} w f(s)\Vert_\infty \notag \\
  &+ Ck \big[e^{-\lambda t_1^u}\sup_{0\leq s\leq T_0} \Big\Vert e^{\lambda s}\frac{w}{\langle v\rangle} g(s) \Big\Vert_\infty  +  e^{-\lambda t_1^u} \sup_{0\leq s \leq T_0} e^{\lambda s}\Vert f(s)\Vert_{L^2}  \big] \Big\} \notag\\
  &\leq  C(\theta) \int_{t_1}^{T_0} \dd s e^{-\nu(T_0-s)}    \times \Big[ e^{-\frac{\nu_0 s}{2}} \Vert w f_0\Vert_\infty + o(1)e^{-\lambda s}\sup_{0\leq s\leq T_0} \Vert e^{\lambda s} wf(s)\Vert_\infty  \notag\\
  & + C(T_0) \big[e^{-\lambda s}\sup_{0\leq s\leq T_0} \Big\Vert e^{\lambda s} \frac{w}{\langle v\rangle}g(s) \Big\Vert_\infty  +  e^{-\lambda s} \sup_{0\leq s \leq T_0} e^{\lambda s}\Vert f(s)\Vert_{L^2}  \big]   \Big]     \notag\\
  &\leq   C(\theta) e^{-\frac{\nu_0 T_0}{2}} \Vert w f_0\Vert_\infty + o(1) e^{-\lambda T_0}\sup_{0\leq s\leq T_0}\Vert e^{\lambda s}w f(s)\Vert_\infty  \notag\\
  & + C(T_0)\big[e^{-\lambda T_0}\sup_{0\leq s\leq T_0} \Big\Vert e^{\lambda s}\frac{w}{\langle v\rangle}g(s) \Big\Vert_\infty  +   e^{-\lambda T_0} \sup_{0\leq s \leq T_0} e^{\lambda s}\Vert f(s)\Vert_{L^2} \big].   \label{u_bdr_bdd}
\end{align}
In the third line we have used Lemma \ref{lemma:k_theta}.

Then we focus on the contribution of \eqref{u_K1} in \eqref{K_1}. First we decompose the $\dd s' $ integral into $\mathbf{1}_{s-s'< \delta} + \mathbf{1}_{s-s'\geq \delta}$. Applying \eqref{k_theta} in Lemma \ref{lemma:k_theta} twice, the contribution of the first term reads
\begin{align}
    & \eqref{u_K1} \mathbf{1}_{s-s'<\delta} \notag \\
   & \leq   \int_{t_1}^{T_0} \dd s e^{-\nu(T_0-s)} \int_{\mathbb{R}^3} \dd u \mathbf{k}_\theta(V_0(s),u) \int^s_{\max\{s-\delta,t_1^u\}} e^{-\nu(u)(s-s')} \dd s' \notag \\
   & \times \int_{\mathbb{R}^3} \dd u' \mathbf{k}_\theta(V_0^u(s'),u')    e^{-\lambda s'} \sup_{0\leq s\leq T_0}\Vert e^{\lambda s} w f(s)\Vert_\infty   \notag \\
   & \leq \int_{t_1}^{T_0} \dd s e^{-\nu(T_0-s)} \int_{\mathbb{R}^3} \dd u \mathbf{k}_\theta(V_0(s),u) e^{-\lambda s} \sup_{0\leq s\leq T_0}\Vert e^{\lambda s} w f(s)\Vert_\infty \notag \\
    & \leq    o(1)e^{-\lambda T_0}\sup_{0\leq s\leq T_0}\Vert e^{\lambda s} w f(s)\Vert_\infty. \label{u_K1_s_small}
\end{align}

Next we decompose the $\dd u$ integral into $\mathbf{1}_{|u|> N \text{ or } |V_0(s)-u|\leq \frac{1}{N}} + \mathbf{1}_{|u|\leq N, \ |V_0(s)-u|>\frac{1}{N}}$. Applying \eqref{k_theta} and \eqref{K_N_small} in Lemma \ref{lemma:k_theta}, the contribution of the first term reads
\begin{align}
   &\eqref{u_K1}\mathbf{1}_{|u|>N \text{ or }|V_0(s)-u| \leq \frac{1}{N}} \notag\\
   & \leq \int_{t_1}^{T_0} \dd s e^{-\nu(T_0-s)} \int_{|u|>N \text{ or }|V_0(s)-u|\leq \frac{1}{N} } \dd u \mathbf{k}_\theta(V_0(s),u) \notag \\
   &\times \int^s_{t_1^u} e^{-\nu(u)(s-s')}  e^{-\lambda s'} \dd s' \sup_{0\leq s\leq T_0}\Vert e^{\lambda s} w f(s)\Vert_\infty \notag \\
   & \leq o(1)\int_{t_1}^{T_0} \dd s e^{-\nu(T_0-s)} e^{-\lambda s} \sup_{0\leq s \leq T_0}\Vert e^{\lambda s}w f(s)\Vert_\infty \notag \\
   & \leq o(1) e^{-\lambda T_0} \sup_{0\leq s\leq T_0}\Vert e^{\lambda s}w f(s)\Vert_\infty. \label{u_K1_u_small}
\end{align}

Next we decompose the $\dd u'$ integral into $\mathbf{1}_{|u'|\geq N \text{ or } |V_0^u(s')-u'|\leq \frac{1}{N}} + \mathbf{1}_{|u'|\leq N, \ |u'-V_0^u(s')|>\frac{1}{N}}$. The contribution of the first term reads
\begin{align}
   &\eqref{u_K1}\mathbf{1}_{|u'|\geq N \text{ or }|u'-V_0^u(s')|\leq\frac{1}{N}} \notag\\
   & \leq o(1)\int_{t_1}^{T_0} \dd s e^{-\nu(T_0-s)} \int_{\mathbb{R}^3} \dd u \mathbf{k}_\theta(V_0(s),u) \int^s_{t_1^u} e^{-\nu(u)(s-s')}  e^{-\lambda s'}\dd s' \sup_{0\leq s\leq T_0}\Vert e^{\lambda s} w f(s)\Vert_\infty   \notag\\
   & \leq   o(1)\int_{t_1}^{T_0} \dd s e^{-\nu(T_0-s)}  e^{-\lambda s} \sup_{0\leq s\leq T_0} \Vert e^{\lambda s} w f(s)\Vert_\infty    \notag\\
   & \leq  o(1) e^{-\lambda T_0} \sup_{0\leq s\leq T_0}\Vert e^{\lambda s}w f(s)\Vert_\infty. \label{u_K1_u_prime_small}
\end{align}

Now we consider the intersection of all other cases, where we have $|u-V_0(s)|>\frac{1}{N}, \ |u|\leq N,$ $s'<s-\delta$ and $|u'|<N, \  |V_0^u(s')-u'|>\frac{1}{N}$. In such case by \eqref{k_N_upper_bdd} we have
$$
\mathbf{k}_\theta(V_0(s),u)w(V_0^u(s'))\mathbf{k}(V_0^u(s'),u')\leq C_N.
$$
We compute such contribution in \eqref{u_K1} as
\begin{align}
    &  \eqref{u_K1}\mathbf{1}_{|u-V_0(s)|>\frac{1}{N}, \ |u|\leq N,\ s'<s-\delta, \ |u'|<N, \ |V_0^u(s')-u'|>\frac{1}{N}} \notag\\
    &  \leq C_N  \int_{t_1}^{T_0} \dd s e^{-\nu (T_0 -s)} \int_{|u|\leq N} \dd u \times \int^{s-\delta}_{t_1^u} e^{-\nu(u)(s-s')} \dd s' \int_{|u'|<N} \dd u'   f(s', X_0^u(s'),u') . \label{u_K1_other}
\end{align}
With $|u'|<N$, we apply the same argument in \eqref{other_case}. We denote 
\[\tilde{\O}_u = \{x\in \mathbb{R}^3| x= X_0(s) - (s-s')u, \ |u'|\leq N, \ X_0^u(s') \in \bar{\O}\}.\]
We define a map
\[\mathcal{M}_u : X_0(s) - (s-s')u \in \tilde{\O}_u \to X_0^u(s') \in \bar{\O}.\]
Similarly to \eqref{omega_i_-} and \eqref{omega_i_+}, we have
\[\tilde{\O}_u = \bigcup_{i=-1-NT_0/L}^{1+NT_0/L} \tilde{\O}_{i}.\]

Then we apply the change of variable $u\to y=X_0(s)-(s-s')u$ with Jacobian 
\[\Big|\det\Big(\frac{\p X_0(s)-(s-s')u}{\p u} \Big) \Big| = (s-s')^3 \geq \delta^3 \]
to derive that
\begin{align}
 & \eqref{u_K1_other} \notag \\
 & \leq     C_{T_0,N,L,\delta,\O} \int_{t_1}^{T_0} \dd s e^{-\nu(T_0-s)} \int_{\tilde{\O}_u} \dd y  \int_{t_1^u}^{s-\delta} e^{-\nu(u)(s-s')} \dd s' \int_{|u'|<N} f(s',\mathcal{M}_u(y),u') \dd u'           \notag \\
  &\leq     C_{T_0,N,L,\delta,\O} \int_{t_1}^{T_0} \dd s e^{-\nu(T_0-s)} \sum_{i=-1-NT_0/L}^{1+NT_0/L}\int_{\tilde{\O}_i} \dd y  \int_{0}^{s-\delta} e^{-\nu(u)(s-s')} \dd s' \int_{|u'|<N} f(s',\mathcal{M}_u(y),u') \dd u'           \notag \\
  & \leq     C_{T_0,N,L,\delta,\O} \int_{t_1}^{T_0} \dd s e^{-\nu(T_0-s)} \sum_{i=-1-NT_0/L}^{1+NT_0/L}\int_{\O} \dd x  \int_{0}^{s-\delta} e^{-\nu(u)(s-s')} \dd s' \int_{|u'|<N} f(s',x,u') \dd u'   \notag \\
  & \leq C_{T_0,N,L,\delta,\O} \int_{t_1}^{T_0} \dd s e^{-\nu(T_0-s)}  \int_{0}^{s-\delta} e^{-\nu(u)(s-s')} \Vert f(s')\Vert_{L^2} \dd s'     \notag \\
  & \leq C_{T_0,N,L,\delta,\O} \sup_{0\leq s\leq T_0} e^{\lambda s}\Vert f(s)\Vert_{L^2} \int_{t_1}^{T_0} \dd s e^{-\nu(T_0-s)}   e^{-\lambda s} \notag \\
  &\leq  C_{T_0,N,L,\delta,\O} e^{-\lambda T_0} \sup_{0\leq s\leq T_0} e^{\lambda s}\Vert f(s)\Vert_{L^2} .  \label{u_K1_other_bdd}
\end{align}
In the third inequality we have applied the change of variable $y \in \tilde{\O}_i \to x=M_u(y)\in \O$ with Jacobian $1$.

Collecting \eqref{u_K1_s_small}, \eqref{u_K1_u_prime_small}, \eqref{u_K1_u_small} and \eqref{u_K1_other_bdd}, we have
\begin{equation}\label{u_K1_bdd}
\eqref{u_K1} \leq o(1) e^{-\lambda T_0} \sup_{0\leq s\leq T_0} \Vert e^{\lambda s}w f(s)\Vert_\infty + C(T_0)  e^{-\lambda T_0}\sup_{0\leq s\leq T_0} e^{\lambda s}\Vert f(s)\Vert_{L^2} .
\end{equation}

By the same computation, we have
\begin{equation}\label{u_K0_bdd}
 \eqref{u_K0} \leq  o(1) e^{-\lambda T_0} \sup_{0\leq s\leq T_0} \Vert e^{\lambda s}w f(s)\Vert_\infty + C(T_0) e^{-\lambda T_0}\sup_{0\leq s\leq T_0} e^{\lambda s}\Vert f(s)\Vert_{L^2}.   
\end{equation}

We combine \eqref{u_initial_bdd}, \eqref{u_g_bdd}, \eqref{u_bdr_bdd}, \eqref{u_K0_bdd} and \eqref{u_K1_bdd} to conclude the estimate for \eqref{K_1}:
\begin{equation}\label{K1_bdd}
\begin{split}
   \eqref{K_1} & \leq C(\theta)e^{-\frac{\nu_0 T_0}{2}} \Vert w f_0\Vert_\infty + o(1)e^{-\lambda T_0}\sup_{0\leq s\leq T_0} \Vert e^{\lambda s}w f(s)\Vert_\infty \notag \\
   &+ C(T_0) e^{-\lambda T_0}\sup_{0\leq s\leq T_0} \Big\Vert e^{\lambda s}\frac{w}{\langle v\rangle}g(s) \Big\Vert_\infty   +C(T_0) e^{-\lambda T_0}\sup_{0\leq s\leq T_0} e^{\lambda s}\Vert f(s)\Vert_{L^2}.
\end{split}
\end{equation}

Similarly, we can have the same estimate for \eqref{K_0} as
\begin{equation}\label{K0_bdd}
\begin{split}
   \eqref{K_0} & \leq C(\theta)e^{-\frac{\nu_0 T_0}{2}} \Vert w f_0\Vert_\infty + o(1)e^{-\lambda T_0}\sup_{0\leq s\leq T_0} \Vert e^{\lambda s}w f(s)\Vert_\infty    \\
    & + C(T_0) e^{-\lambda T_0}\sup_{0\leq s\leq T_0} \Big\Vert e^{\lambda s} \frac{w}{\langle v\rangle} g(s) \Big\Vert_\infty +C(T_0)  e^{-\lambda T_0}\sup_{0\leq s\leq T_0} e^{\lambda s}\Vert f(s)\Vert_{L^2}.
\end{split}
\end{equation}

Last we collect \eqref{initial_g_bdd}, \eqref{f_bdr_bdd}, \eqref{K1_bdd} and \eqref{K0_bdd} to conclude that
\begin{align}
 &w(v) |f(T_0,x,v)|   \notag\\
 &   \leq C(\theta)e^{-\frac{\nu_0 T_0}{2}} \Vert w f_0\Vert_\infty +o(1)e^{-\lambda T_0}\sup_{0\leq s\leq T_0} \Vert e^{\lambda s}w f(s)\Vert_\infty     \label{C_theta} \\
  &  + C(T_0) e^{-\lambda T_0}\sup_{0\leq s\leq T_0}\Big\Vert e^{\lambda s}\frac{w}{\langle v\rangle }g(s) \Big\Vert_\infty +C(T_0) e^{-\lambda T_0}\sup_{0\leq s\leq T_0} e^{\lambda s}\Vert f(s)\Vert_{L^2}.  \notag
\end{align}

With the weight $w(v)=e^{\theta |v|^2}$, we bound the $L^2$ norm by Proposition \ref{prop:l2} as
\begin{align}
   \sup_{0\leq s\leq T_0}e^{\lambda s}\Vert f(s)\Vert_{L^2} & \leq C(T_0)\big[\Vert f_0\Vert_{L^2} + \int_0^{T_0} \Vert e^{\lambda s}g(s)\Vert_{L^2}\dd s \big] \notag \\
   &\leq C(T_0)\big[ \Vert wf_0\Vert_\infty +  \sup_{0\leq s\leq T_0} \Big\Vert e^{\lambda s} \frac{w}{\langle v\rangle} g(s)\Big\Vert_\infty \big]. \label{l2_term}
\end{align}

For given $0\leq t<\infty$, we denote
\begin{align*}
    \mathcal{R}_t := \Vert wf_0\Vert_\infty + \sup_{0\leq s\leq t}\Big\Vert e^{\lambda s} \frac{w}{\langle v\rangle} g(s) \Big\Vert_\infty.
\end{align*}

Recall that $C(\theta)$ in \eqref{C_theta} does not depend on $T_0$. We choose $T_0$ to be large enough such that  $C(\theta )e^{-\frac{\nu_0 T_0}{2}} < e^{-\frac{\nu_0 T_0}{4}}$. Then we further have
\begin{align}
   \Vert  wf(T_0)\Vert_\infty  &  \leq  e^{-\frac{\nu_0 T_0}{4}} \Vert wf_0\Vert_\infty    + o(1)e^{-\lambda T_0}\sup_{0\leq s\leq T_0} \Vert e^{\lambda s}w f(s)\Vert_\infty \notag \\
   &+  C(T_0) e^{-\lambda T_0} \sup_{0\leq s\leq T_0} \Big\Vert e^{\lambda s} \frac{w}{\langle v\rangle} g(s) \Big\Vert_\infty + C(T_0)e^{-\lambda T_0}\sup_{0\leq s\leq T_0} e^{\lambda s}\Vert f(s)\Vert_{L^2}.  \label{est_T0}
\end{align}

For $0\leq t\leq T_0$, with the same choice of $k=C_1 T_0^{5/4}$, it is straightforward to apply the same argument for $e^{\lambda t}w(v)|f(t,x,v)|$ to have:
\begin{align}
  \Vert wf(t)\Vert_\infty  & \leq C(\theta)e^{-\frac{\nu_0 t}{2}}\Vert wf_0\Vert_\infty + o(1) e^{-\lambda t} \sup_{0\leq s\leq t}\Vert e^{\lambda s} wf(s)\Vert_\infty \notag\\
  & + C(T_0)e^{-\lambda t} \sup_{0\leq s\leq t} \Big\Vert e^{\lambda s} \frac{w}{\langle v\rangle} g(s) \Big\Vert_\infty + C(T_0)e^{-\lambda t}\sup_{0\leq s\leq t} e^{\lambda s}\Vert f(s)\Vert_{L^2}. \label{est_t}
\end{align} 

For $t=mT_0$, we apply \eqref{est_T0} to have
\begin{align}
  &\Vert wf(mT_0)\Vert_\infty  \notag\\
  &  \leq e^{-\frac{\nu_0 T_0}{4}} \Vert wf((m-1)T_0)\Vert_\infty + C(T_0)e^{-\lambda T_0} \sup_{0 \leq s\leq T_0} \Big\Vert e^{\lambda s}\frac{w}{\langle v\rangle}g((m-1)T_0+s) \Big\Vert_\infty   \notag\\
  & + o(1)e^{-\lambda T_0} \sup_{0 \leq s\leq T_0} \Vert e^{\lambda s }w f((m-1)T_0+s) \Vert_\infty + C(T_0) e^{-\lambda T_0} \sup_{0\leq s\leq T_0}e^{\lambda s}\Vert f((m-1)T_0+s)\Vert_{L^2}  \notag\\
  & \leq e^{-\frac{\nu_0 T_0}{4}} \Vert wf((m-1)T_0)\Vert_\infty +o(1)e^{-\lambda m T_0} \sup_{0 \leq s\leq mT_0} \Vert e^{\lambda s }w f(s) \Vert_\infty + C(T_0) e^{-\lambda m T_0} \mathcal{R}_{mT_0} \notag\\
  & \leq e^{-2\frac{\nu_0 T_0}{4}} \Vert wf((m-2)T_0)\Vert_\infty \notag \\
  &+  e^{-\lambda m T_0} \Big[o(1)\sup_{0\leq s\leq mT_0} \Vert e^{\lambda s}w f(s)\Vert_\infty + C(T_0) \mathcal{R}_{mT_0} \Big]\times \big[1 + e^{-\frac{(\nu_0-4\lambda) T_0}{4}} \big] \notag\\
  & \leq \cdots \leq e^{-\frac{m\nu_0T_0}{4}} \Vert wf_0\Vert_\infty \notag\\
  & + e^{-\lambda m T_0} \Big[o(1)\sup_{0\leq s\leq mT_0} \Vert e^{\lambda s}w f(s)\Vert_\infty + C(T_0) \mathcal{R}_{mT_0} \Big]\times \sum_{i=0}^{m-1}  e^{-\frac{m(\nu_0-4\lambda)T_0}{4}} \notag\\
  & \leq  o(1)C(\nu_0)e^{-\lambda m T_0} \sup_{0\leq s\leq mT_0}\Vert e^{\lambda s}wf(s)\Vert_\infty + C(T_0)e^{-\lambda m T_0} \mathcal{R}_{mT_0}. \label{mT0}
\end{align}
In the fourth line we have applied the same computation as \eqref{l2_term} to the $L^2$ term.

For any $t>0$, we can choose $m$ such that $mT_0\leq t\leq (m+1)T_0$. With $t=mT_0+s$, $0\leq s\leq T_0$, we apply \eqref{est_t} to have
\begin{align}
  & \Vert wf(t) \Vert_\infty  = \Vert wf(mT_0 + s)\Vert_\infty \notag\\
  & \leq C(\theta)e^{\frac{-\nu_0 s}{2}}\Vert wf(mT_0)\Vert_\infty  + o(1)e^{-\lambda s} \sup_{0\leq s'\leq s} \Vert e^{\lambda s'}wf(mT_0+s')\Vert_\infty \notag\\
  &+ C(T_0) e^{-\lambda s}\sup_{0\leq s'\leq s}   \Big\Vert e^{\lambda s'}\frac{w}{\langle v\rangle}g(mT_0+s') \Big\Vert_\infty + C(T_0) e^{-\lambda s} \sup_{0\leq s'\leq s}e^{\lambda s'}\Vert f(mT_0+s')\Vert_{L^2} \notag\\
  & \leq o(1)C(\nu_0,\theta)e^{-\lambda (m T_0+s)} \sup_{0\leq s\leq mT_0}\Vert e^{\lambda s}wf(s)\Vert_\infty + C(T_0)e^{-\lambda (m T_0+s)} \mathcal{R}_{mT_0+s} \notag\\
  & \leq o(1)e^{-\lambda t} \sup_{0\leq s\leq t}\Vert e^{\lambda s}wf(s)\Vert_\infty + C(T_0) e^{-\lambda t} \mathcal{R}_t. \label{f_t_bdd}
\end{align}
In the fourth line we have applied \eqref{mT0} and \eqref{l2_term} to the $L^2$ term.

Since \eqref{f_t_bdd} holds for all $t$, we conclude that
\begin{align*}
  e^{\lambda t}\Vert wf(t)\Vert_\infty  & \leq C(T_0) e^{-\lambda t}\big[\Vert wf_0\Vert_\infty + \sup_{0\leq s\leq t}\Big\Vert e^{\lambda s}\frac{w}{\langle v\rangle}g(s)\Big\Vert_\infty\big].
\end{align*}
We conclude the proof of Proposition \ref{prop:linfty}.  
\end{proof}

\begin{proof}[\textbf{Proof of Theorem \ref{thm:linfty}}]
We consider the following iteration sequence:
\begin{equation}\label{iter_eqn}
\p_t f^{\ell+1} + v\cdot \nabla_x f^{\ell+1} + \mathcal{L} f^{\ell+1} = \Gamma(f^\ell,f^\ell).
\end{equation}
Proposition \ref{prop:linfty} gives the a priori $L^\infty$ estimate to \eqref{iter_eqn}. For the existence of such $L^\infty$ solution in Proposition \ref{prop:linfty}, we refer to \cite{G}, \cite{EGKM} for the construction using sequential argument.

We apply Proposition \ref{prop:linfty} to have
\begin{align*}
  \sup_{0\leq s\leq t} \Vert e^{\lambda s} wf^{\ell+1}\Vert_\infty  & \leq C\Vert wf_0\Vert_\infty + C\Big[ \sup_{0\leq s\leq t} \Vert e^{\lambda s}w f^{\ell}\Vert \Big]^2  .
\end{align*}
Here we have used Lemma \ref{lemma:gamma} to have
\begin{align*}
 \sup_{0\leq s\leq t}\Big\Vert  e^{\lambda s} \frac{w}{\langle v\rangle} \Gamma(f^\ell,f^\ell) \Big\Vert_\infty  &   \leq \sup_{0\leq s\leq t} \Big\Vert \frac{w}{\langle v\rangle}\Gamma(e^{\lambda s}f^{\ell},e^{\lambda s}f^{\ell}) \Big\Vert_\infty \\
 &\leq \Big[ \sup_{0\leq s\leq t} \Big\Vert e^{\lambda s}w f^{\ell} \Big\Vert_\infty \Big]^2 .
\end{align*}

When the initial condition satisfies \eqref{initial_condition} for $\delta$ small enough such that $2C\delta \ll 1$, then when $\Vert e^{\lambda s}w f^\ell\Vert_\infty \leq 2C\delta$, we have
\begin{align*}
    \sup_{0\leq s\leq t} \Vert e^{\lambda s} wf^{\ell + 1} \Vert_\infty \leq C\delta + 4C\delta^2 \leq 2C\delta.
\end{align*}
Hence we conclude the uniform-in-$\ell$ estimate:
\begin{equation*}
\sup_{\ell} \sup_{0\leq s\leq t} \Vert e^{\lambda s} w f^{\ell} \Vert_\infty \leq 2C\delta.
\end{equation*}

By taking the difference $f^{\ell+1}-f^\ell$ and repeating the same argument, we have
\begin{align*}
  \sup_{0\leq s\leq t}\Vert e^{\lambda s} w (f^{\ell+1}-f^\ell)\Vert_\infty  &  \leq \big[  \sup_{0\leq s\leq t} \Vert e^{\lambda s} wf^{\ell + 1} \Vert_\infty +  \sup_{0\leq s\leq t} \Vert e^{\lambda s} wf^{\ell} \Vert_\infty\big] \\
  & \times  \sup_{0\leq s\leq t}\Vert e^{\lambda s} w (f^{\ell}-f^{\ell - 1})\Vert_\infty. 
\end{align*}
Thus $f^{\ell}$ is a Cauchy sequence, and we construct a solution $f$ such that $\Vert wf(t)\Vert_\infty \leq 2e^{-\lambda t}C\delta$. The uniqueness and non-negativity of solutions follow from the standard way.
\end{proof}

\noindent {\bf Acknowledgment:}\,
The research of Renjun Duan was partially supported by the General Research Fund (Project No.~14303321) from RGC of Hong Kong and a Direct Grant from CUHK. Hongxu Chen thanks Chanwoo Kim for helpful discussion.

\medskip

\noindent{\bf Conflict of Interest:} The authors declare that they have no conflict of interest.

\bibliographystyle{siam}

\end{document}